\documentclass[UTF-8,reqno]{amsart}
\usepackage{enumerate, bbm}
\usepackage{amssymb,url,color, booktabs}
\usepackage{tikz}
\usepackage{mathrsfs}
\usepackage{dutchcal}
\usepackage{threeparttable}
\usepackage{bm}
\usepackage{pgfplots}
\usepackage{graphicx}
\usepackage{array}
\usepackage{booktabs}
\usepackage{diagbox} 
\usepackage{caption}
\usepackage{xcolor}
\usepackage{colortbl}

\usepackage{color}
\usepackage[colorlinks=true]{hyperref}
\hypersetup{
    linkcolor=blue,          
    citecolor=red,        
    filecolor=blue,      
    urlcolor=cyan
}

\usepackage{color}
\usepackage{ulem}

\definecolor{MyDarkBlue}{cmyk}{0.8,0.3,0.8,0.4}
\definecolor{yellow}{rgb}{0.99,0.99,0.70}
\definecolor{white}{rgb}{1.0,1.0,1.0}
\definecolor{black}{rgb}{0.00,0.00,0.00}

\numberwithin{equation}{section}
\def\theequation{\arabic{section}.\arabic{equation}}
\newcommand{\be}{\begin{eqnarray}}
\newcommand{\ee}{\end{eqnarray}}
\newcommand{\ce}{\begin{eqnarray*}}
\newcommand{\de}{\end{eqnarray*}}
\newtheorem{theorem}{Theorem}[section]
\newtheorem{lemma}[theorem]{Lemma}
\newtheorem{remark}[theorem]{Remark}
\newtheorem{definition}[theorem]{Definition}
\newtheorem{proposition}[theorem]{Proposition}

\newtheorem{corollary}[theorem]{Corollary}

\def\eps{\varepsilon}

\def\e{\mathrm{e}}

\def\p{\partial}

\def\[{{\Big[}}
\def\]{{\Big]}}
\def\<{{\langle}}
\def\>{{\rangle}}
\def\({{\big(}}
\def\){{\big)}}

\def\dif{{\mathord{{\rm d}}}}

\def\min{{\mathord{{\rm min}}}}

\def\bb2{{\boldsymbol{2}}}
\def\no{\nonumber}
\def\={&\!\!=\!\!&}

\def\bC{{\mathbf C}}

\def\sA{{\mathcal A}}

\def\cC{{\mathcal C}}
\def\cD{{\mathcal D}}

\def\cJ{{\mathcal J}}

\def\cP{{\mathcal P}}

\def\cW{{\mathcal W}}

\def\mE{{\mathbb E}}

\def\mI{{\mathbb I}}

\def\mN{{\mathbb N}}

\def\mP{{\mathbb P}}

\def\mR{{\mathbb R}}

\def\b1{{\mathbbm 1}}

\def\sA{{\mathscr A}}

\def\sI{{\mathscr I}}

\def\sL{{\mathscr L}}

\def\sU{{\mathscr U}}

\def\geq{\geqslant}
\def\leq{\leqslant}

\def\le{\leqslant}

\def\eps{\varepsilon}

\def\e{\mathrm{e}}

\def\p{\partial}

\def\[{{\Big[}}
\def\]{{\Big]}}
\def\<{{\langle}}
\def\>{{\rangle}}

\def\dif{{\mathord{{\rm d}}}}

\def\min{{\mathord{{\rm min}}}}

\def\no{\nonumber}
\def\={&\!\!=\!\!&}
\def\bt{\begin{theorem}}
\def\et{\end{theorem}}
\def\bl{\begin{lemma}}
\def\el{\end{lemma}}
\def\br{\begin{remark}}
\def\er{\end{remark}}
\def\bd{\begin{definition}}
\def\ed{\end{definition}}
\def\bp{\begin{proposition}}
\def\ep{\end{proposition}}
\def\bc{\begin{corollary}}
\def\ec{\end{corollary}}

\def\geq{\geqslant}
\def\leq{\leqslant}

\def\le{\leqslant}

 \def\R{\mathbb R}
\def\ff{\frac} \def\R{\mathbb R}  \def\ff{\frac}

\def\<{\langle} \def\>{\rangle}

\def\bb2{{\boldsymbol{2}}}
\def\bbb1{\boldsymbol{1}}

\def\no{\nonumber}
\def\={&\!\!=\!\!&}

\allowdisplaybreaks

\begin{document}

\title[$\cW_{\bf d}$-convergence rate of EM schemes for invariant measures]
{
$\cW_{\bf d}$-convergence rate of EM schemes for invariant measures of supercritical stable SDEs}

\author[P.~Chen]{Peng Chen}
\address[P.~Chen]{School of Mathematics, Nanjing University of Aeronautics and
Astronautics, Nanjing, 211106, China}
\email{chenpengmath@nuaa.edu.cn}

\author[L.~Xu]{Lihu Xu}
\address[L.~Xu]{1. Department of Mathematics, Faculty of Science and Technology, University of Macau, Macau S.A.R., China. 2. Zhuhai UM Science \& Technology Research Institute, Zhuhai, China}
\email{lihuxu@um.edu.mo}

\author[X.~Zhang]{Xiaolong Zhang}
\address[X.~Zhang]{School of Mathematics and Statistics, Beijing Institute of Technology, Beijing, 100080, P.R. China}
\email{zhangxl.math@bit.edu.cn}

\author[X.~Zhang]{Xicheng Zhang}
\address[X.~Zhang]{School of Mathematics and Statistics, Beijing Institute of Technology, Beijing 100081, China\\
		Faculty of Computational Mathematics and Cybernetics, Shenzhen MSU-BIT University, 518172 Shenzhen, China}
\email{XichengZhang@gmail.com}

\maketitle

{\small \noindent {\bf Abstract:}
By establishing the regularity estimates for nonlocal Stein/Poisson equations under $\gamma$-order
H\"older and dissipative conditions on the coefficients,
we derive the $\cW_{\bf d}$-convergence rate for the Euler-Maruyama schemes applied to the invariant measure of SDEs driven by multiplicative $\alpha$-stable noises with $\alpha \in (\frac{1}{2}, 2)$, where $\cW_{\bf d}$ denotes the Wasserstein metric with ${\bf d}(x,y)=|x-y|^\gamma\wedge 1$
and $\gamma \in ((1-\alpha)_+, 1]$.
~\\


\bigskip

\section{Introduction}
The study of sampling random variables from a given probability distribution has broad applications across various fields such as social sciences, engineering, agriculture, ecology and medicine (see \cite{L10,RK16}). A widely used method for this task is the Langevin Monte Carlo (LMC), which approximates the target distribution through Langevin diffusion processes. Numerous works have examined the performance of LMC algorithms (see \cite{KSL22,M-T2,NSR19} and references therein). Specifically, consider a probability distribution density $\pi(x) \propto \e^{-U(x)}$, where $U:\mR^d \to \mR$ is a potential function. The associated stochastic differential equation (SDE) is given by:
\begin{align}\label{SDE6}
\dif X_t = -\nabla U(X_t)\dif t + \sqrt{2}\dif B_t, \ \ X_0 = x,
\end{align}
where $(B_t)_{t \geq 0}$ is a $d$-dimensional standard Brownian motion, and $X$ represents a particle's trajectory in the potential field $U(x)$. Under suitable conditions on $U$, 
SDE \eqref{SDE6} has a unique invariant probability measure $\mu(\dif x):= \pi(x)\dif x$, and the law of the solution $X_t$ converges exponentially to $\mu$ as $t \to \infty$ (cf. \cite{BEL18,M-T2}). The Euler-Maruyama (EM) scheme is commonly used to numerically approximate solutions to SDE \eqref{SDE6} and generate samples from $\mu$ (cf. \cite{BEL18}).
Given step size $\eta>0$, let $\mu_\eta$ denote the invariant measure corresponding to this scheme.
This raises two important questions:
\begin{itemize}
\item Can $\mu_\eta$ approximate $\mu$ as $\eta \to 0$?
\item If yes, what is the optimal rate of convergence?
\end{itemize}
Recent studies \cite{Dal17, DR20, FSX19} have established bounds on the Wasserstein distance between $\mu_\eta$ and $\mu$, achieving convergence rates of order $\eta^{1/2}$ or $\eta^{1/2} \log (1/\eta)$  under various assumptions. For more detailed results on  the Langevin sampling and the EM scheme, it is referred to \cite{EGZ19,FG16,FSX19,PP20}.

In statistical physics, the Langevin diffusion $X_t$ described by SDE \eqref{SDE6} represents the position of a particle at time $t$, which is subject to random forces. According to the central limit theorem (CLT), if these random impulses exhibit \textit{infinite} variance, their sum may converge to an $\alpha$-stable distribution with $\alpha \in (0, 2)$.
More precisely, consider a sequence of i.i.d. random variables $X_1, X_2, \dots$ with the following common distribution:
$$
\mathbb{P}(X_1>x) = \mathbb{P}(X_1<-x) =x^{-\alpha}/2, \quad x \geq 1,
$$
where $\alpha \in (0,2)$. For the random walk $S_n = X_1 + \dots + X_n$, the CLT (see \cite[p.~158]{Du10})
asserts that
$$
S_n / n^{1/\alpha} \Rightarrow Y,
$$
where $Y$ is an $\alpha$-stable random variable, and $\Rightarrow$ denotes weak convergence. Heavy-tailed distributions, like $\alpha$-stable distributions, often arise in various stochastic systems. For example, they are used to model inputs in computer and communications networks, represent risk processes, and occur naturally in models of epidemiological spread. Such distributions are backed by statistical evidence in numerous fields, including physics, geoscience and economics (see \cite{FKZ11} and references therein).
For further reading on heavy-tailed distributions and their applications in insurance and finance, it is referred to \cite{BGT87, EKM97}. In a different context, for $d=1$ and $U(x) = |x|^\beta$ in \eqref{SDE6}, Roberts and Tweedie \cite{RT96} showed that the diffusion process defined by SDE \eqref{SDE6} is exponentially ergodic if and only if $\beta \geq 1$. Therefore, simulating heavy-tailed distributions using a continuous Langevin diffusion \eqref{SDE6} is not appropriate for $\beta < 1$, as the process fails to capture the heavy-tailed nature of such distributions.

In this paper, we study the following $d$-dimensional SDE driven by an $\alpha$-stable process:
\begin{align}\label{sde:t}
\dif X_t = b(X_t) \dif t + \sigma(X_{t-}) \dif L^{(\alpha)}_t,
\end{align}
where $\alpha \in (0,2)$, and $\big\{L^{(\alpha)}_t, t \geq 0\big\}$ is a rotationally invariant $\alpha$-stable process in $\mathbb{R}^d$ with L\'evy measure
$$
\nu(\dif z) = \dif z/|z|^{d+\alpha}.
$$
Throughout this paper, we assume that $b:\mathbb{R}^d \to \mathbb{R}^d$ and $\sigma:\mathbb{R}^d \to \mathbb{R}^{d \times d}$ satisfy the following conditions:
\begin{enumerate}
\item[$(\mathbf{H}^\gamma)$] For some $\gamma \in ((1-\alpha)_+, 1]$ and $\kappa_0 > 1$, it holds that for all $x, y, \xi \in \mathbb{R}^d$,
$$
\kappa_0^{-1}|\xi| \leq |\sigma(x)\xi| \leq \kappa_0 |\xi|,\ \ \|\sigma(x) - \sigma(y)\| \leq \kappa_0 |x - y|^\gamma
$$
and
$$
|b(0)| \leq \kappa_0,\ \ |b(x) - b(y)| \leq \kappa_0 \left(|x - y|^\gamma + |x - y|\right).
$$
\end{enumerate}

In recent years, SDE \eqref{sde:t} has been extensively studied, see \cite{ CSZ18, CZZ21,MZ22,WXZ15, ZZ23} for references. In particular, under $(\mathbf{H}^\gamma)$, for each starting point $X_0 = x \in \mathbb{R}^d$, SDE \eqref{sde:t} admits a unique weak solution $X^x_t = X_t$ (see \cite{CZZ21}). Let $(P_t)_{t\geq 0}$ be the semigroup associated with $X^x_t$, that is, for each $f \in L^\infty(\mathbb{R}^d)$,
\begin{align}\label{semi}
 P_t f(x) := \mathbb{E} f(X^x_t).
\end{align}
For $f \in {\bf C}^2(\mathbb{R}^d)$, by It\^o's formula, we have
\begin{align}\label{Duh}
\partial_t P_t f(x) = P_t \sL f(x),\ \ t > 0,\ \ x \in \mathbb{R}^d,
\end{align}
where ${\bf C}^2(\mathbb{R}^d)$ consists of all $C^2$-functions with bounded derivatives up to $2$-order,
and $\sL$ is the generator of the semigroup $(P_t)_{t \geq 0}$, given by
\begin{align}\label{Def1}
\sL f(x) := \sL^{(\alpha)}_{\sigma(x)} f(x) + b(x) \cdot \nabla f(x),
\end{align}
where, for a $d \times d$ matrix $A$,
$$
\sL^{(\alpha)}_A f(x) = \int_{\mathbb{R}^d} [f(x + Az) - f(x) - \mathbf{1}_{|z| \leq 1} Az \cdot \nabla f(x)]\dif z/|z|^{d+\alpha}.
$$

Let $\phi_\eps(x) = \eps^{-d} \phi(x/\eps)$, $\eps\in(0,1)$, where $\phi$ is a smooth probability density function with compact support. Let $\theta_t(x)$ solve the following regularized ODE:
\begin{align*}
\dot\theta_t(x) = (b*\phi_{t^{1/\alpha}})(\theta_t(x)),\ \ \theta_0(x) = x.
\end{align*}
Under $(\mathbf{H}^\gamma)$, by \cite[Theorem 1.1]{MZ22}, for each $t > 0$, $X^x_t$ admits a density $p_t(x,y)$ so that
\begin{align}\label{DD20}
\partial_t p_t(x,y) = \sL p_t(\cdot,y)(x),
\end{align}
and for each $T > 0$, there is a $C = C(T, \kappa_0, d, \alpha, \gamma, \phi) > 0$ such that for any $t\in(0,T)$ and $x,y\in\mR^d$,
\begin{align}\label{TG:ep}
p_t(x,y) \asymp_C t(t^{\frac{1}{\alpha}} + |\theta_t(x)-y|)^{-d-\alpha},
\end{align}
and
\begin{align}\label{Grad01}
t|\Delta^{\alpha/2} p_t(\cdot,y)(x)| + t^{\frac1\alpha}|\nabla_x p_t(x,y)| \lesssim_C p_t(x,y),
\end{align}
where $\nabla_x$ is the gradient operator in the $x$-direction.
Note that \eqref{DD20} implies that for $f \in L^\infty(\mathbb{R}^d)$,
\begin{align*}
\partial_t P_t f(x) = \sL P_t f(x),\ \ t > 0,\ \ x \in \mathbb{R}^d.
\end{align*}

On the other hand, under the following dissipative assumption: for some
$\kappa_1, \kappa_2 > 0$,
\begin{align}\label{Dis}
\langle x, b(x) \rangle \leq -\kappa_1 |x|^2 + \kappa_2,
\end{align}
it is well-known that SDE \eqref{sde:t} has a unique invariant probability measure $\mu$ (see \cite{ZZ23}). This naturally raises the question of whether $\mu$ can be approximated using a computable algorithm, which generates random vectors whose distributions are close to $\mu$. A common approximation scheme is the EM algorithm. Specifically, let $\eta \in (0,1)$ be the step size, and consider the iteration
\begin{align}\label{EM1-1}
Y^x_{k+1} = Y^x_k + \eta b(Y^x_k) + \sigma(Y^x_k) \left(L^{(\alpha)}_{\eta(k+1)}- L^{(\alpha)}_{\eta k}\right), \quad k \geq 0,\ \ Y^x_0=x.
\end{align}
It is well established that the linear interpolation of $Y^x_k$ converges to $X^x_t$ (see \cite{PT97}).
For further details on the EM scheme, we refer to \cite{BT96,BMY09,HL18,KS19}.

Under assumption \eqref{Dis}, it can be shown that the Markov chain $(Y^x_k)_{k \geq 0}$ admits a unique invariant distribution $\mu_\eta$ (see Lemma \ref{W1.1-1L} below). Two main approaches are used to analyze the Wasserstein distance between invariant measures of two Markov processes: the Stein method and the Lindeberg principle (see \cite{AH19,CDRX22,CSX22+}). Both methods require sufficient regularity for at least one of the corresponding semigroups. In this paper, one of our primary objectives is to prove that $\mu_\eta$ converges to $\mu$ as $\eta \to 0$ under certain Wasserstein metrics, and furthermore, to provide an explicit convergence rate using the Stein method.

\subsection{Main Results}

Before presenting our main results, we introduce some notations and basic ideas used below.
Let ${\bf d}$ be a distance on $\mR^d$, and let $\cP_{\bf d}(\mR^d)$ denote the space of all probability measure $\mu$ that satisfies
\begin{align*}
\int_{\mR^d}{\bf d}(x,0) \mu(\dif x)<\infty. \end{align*}
Let $\cW_{\bf d}$ represent the Wasserstein metric on $\cP_{\bf d}(\mR^d)$, defined as
\begin{align*}
\cW_{\bf d}(\mu_1,\mu_2):=\ \inf_{(X,Y)\in\Pi(\mu_1,\mu_2)}\mathbb{E} {\bf d}(X,Y),
\end{align*}
where $\Pi(\mu_1,\mu_2)$ is the set of all coupling realizations of $\mu_1,\mu_2$.
For any $\mu \in \cP_{\bf d}(\mR^d)$ and $f:\mR^d\to\mR$
with $|f(x)|\leq C(1+{\bf d}(x,0))$, we denote
\begin{align*}
\mu(f):=\int_{\mR^d}f(x) \mu(\dif x).
\end{align*}
Using the Kantorovich-Rubinstein duality (see \cite[p.60]{V09}), we have
\begin{align}\label{jsy5}
\cW_{\bf d}(\mu_1,\mu_2)=\sup_{g \in {\rm Lip}({\bf d})}|\mu_1(g)-\mu_2(g)|, \end{align}
where
\begin{align*}
{\rm Lip}({\bf d}):=\big\{g: \R^d\to\R: |g(x)-g(y)| \le {\bf d}(x,y)\big\}.
\end{align*}
For given $g\in {\rm Lip}({\bf d})$, letting $\bar{g}=:g-\mu(g)$, we consider the Stein/Poisson equation\footnote{The Poisson equation is also called the Stein equation due to its central role in Stein's method.} defined as follows:
\begin{align}\label{Ste}
\sL f(x)=\mu(g)-g(x)= -\bar g(x).
\end{align}
Suppose that \eqref{Ste} has a regular enough solution $f$. We can express $\cW_{\bf d}(\mu,\mu_\eta)$ as
\begin{align}\label{DP2}
\cW_{\bf d}(\mu,\mu_\eta)\overset{\eqref{jsy5}}=\sup_{g \in {\rm Lip}({\bf d})}|\mu(g)-\mu_\eta(g)| =\sup_{g \in {\rm Lip}({\bf d})}|\mu_\eta(\bar g)| \overset{\eqref{Ste}}=\sup_{g \in {\rm Lip}({\bf d})}|\mu_\eta(\sL f)|.
\end{align}
For fixed $x\in\mR^d$, we introduce the following auxiliary random process:
\begin{align*}
Z^x_t=x+t b(x)+\sigma(x)L^{(\alpha)}_t.
\end{align*}
By Itô's formula, we have
\begin{align}\label{DS9}
\mE f(Z^x_t)=f(x)+\int^t_0\mE\sL_x f(Z^x_s)\dif s,
\end{align}
where
\begin{align*}
 \sL_x f(y):=b(x)\cdot\nabla f(y)+\sL^{(\alpha)}_{\sigma(x)}f(y).
\end{align*}
Thus,
\begin{align*}
\sL f(x)=\frac{1}{\eta}\left[\mE f(Z^x_\eta)-f(x)\right]+\frac1\eta\int^\eta_0\mE[\sL_x f(x)-\sL_x f(Z^x_s)]\dif s.
\end{align*}
Let $\mu_\eta\in \cP_{\bf d}(\mR^d)$ be the stationary distribution of Markov chain $(Y^x_k)_{k\geq 0}$
defined in \eqref{EM1-1}, that is, for $\xi\sim\mu_\eta$, independent of $L^{(\alpha)}_\cdot$,
\begin{align*}
Z^{\xi}_\eta\sim \xi\sim \mu_\eta\in \cP_{\bf d}(\mR^d),
\end{align*}
where $\sim$ denotes the same distribution.
In particular,
\begin{align}\label{FD7}
\mu_\eta(\sL f)=\mE(\sL f(\xi))=\frac1\eta\int^\eta_0\mE\left[\sL_{\xi} f(\xi)-\sL_{\xi} f(Z^{\xi}_s)\right]\dif s.
\end{align}
To show the convergence rate, it suffices to show certain regularity estimates for $f$.

Let $\kappa_1,\kappa_2>0$ and $\kappa_0,\gamma$ be the same as in $(\mathbf{H}^\gamma)$.
Throughout this paper we use the parameter set
\begin{align}\label{The}
\Theta:=(d,\alpha,\kappa_0,\kappa_1,\kappa_2,\gamma).
\end{align}
Now, we are ready to present our main results.

\bt\label{main}
Let $\alpha\in(1,2)$ and ${\bf d}(x,y):=|x-y|$. Suppose that $(\mathbf{H}^1)$ and
$$
\<x-y,b(x)-b(y)\>\leq-\kappa_1|x-y|^2+\kappa_2,\ \ \forall x,y\in\mR^d.
$$
Then there are $\eta_0=\eta_0(\Theta)>0$ and  $C=C(\Theta)>0$ such that for all $\eta\in(0,\eta_0)$,
$$
\cW_{\bf d}(\mu,\mu_\eta)\leq C\eta^{1/\alpha}.
$$
\et

The next result considers weaker H\"older regualrity and dissipativity assumption.
\bt\label{main1}
Let $\alpha\in(\frac12,2)$, $\gamma\in((1-\alpha)_+,1]$ and ${\bf d}(x,y):=|x-y|^\gamma\wedge 1$.
Suppose that $(\mathbf{H}^\gamma)$ and
$$
\<x,b(x)\>\leq-\kappa_1|x|^2+\kappa_2,\ \ \forall x\in\mR^d.
$$
For any $\eps\in(0,2\alpha-1)$, there are $\eta_0=\eta_0(\Theta)>0$ and  $C=C(\Theta,\eps)>0$ such that for all $\eta\in(0,\eta_0)$,
$$
\cW_{\bf d}(\mu,\mu_\eta)\leq C \Big(\eta^{\gamma\wedge(2\alpha-1-\eps)}
\b1_{\alpha\in(\frac12,1]}+\eta^{\gamma/\alpha}\b1_{\alpha > 1}\Big).
$$
\et
\br\rm
At this stage, it seems that our method is not applicable to the case $\alpha \in (0,\frac{1}{2}]$. For further discussion on this limitation, see Remark \ref{jys1} below.
Moreover, the convergence rates established in Theorem \ref{main} and Theorem \ref{main1} may not be sharp, see Theorem \ref{syy2} below for the precise convergence rate of 
$\cW_{\bf d}(\mu,\mu_\eta)$ for OU processes. Obviously, the regularity of $b$ could affect the convergence rate.
\er
\begin{remark}\rm
Although we focus exclusively on EM scheme \eqref{EM1-1}, we can also extend our results to the Pareto-type EM scheme studied in \cite{CDRX22}, which is defined as
\begin{align*}
Y_{k+1} = Y_{k} + \eta b(Y_{k}) + (\eta/\beta)^{\frac{1}{\alpha}} \sigma(Y_{k}) Z_{k+1}, \quad k \geq 0,
\end{align*}
where $\beta = \alpha \Gamma (\frac{d}{2}) /(2\pi^{\frac{d}{2}})$ is the
`inverse temperature' parameter (see \cite{NSR19}) and $\{Z_k, k \in \mathbb{N}\}$ is a sequence of $d$-dimensional independent random vectors drawn from the Pareto distribution with the density function
\begin{align*}
p(z) \propto |z|^{-\alpha - d} \b1_{|z| \geq 1}(z).
\end{align*}
Since the proof procedures are similar to those of EM scheme \eqref{EM1-1}, we do not elaborate on this case further in this paper. We mention that in economics, Pareto distributions and their generalizations serve as flexible models for income distributions and other social and economic phenomena (see \cite{A15}).
\end{remark}

Our main contribution of this paper can be summarized as follows:

\begin{itemize}
    \item We establish the well-posedness and regularity properties for the Stein/Poisson equation
    \eqref{Ste} in certain weighted H\"older spaces with logarithmic growth.
     It is worth noting that our analysis includes cases with multiplicative noise and unbounded drift terms.

   \item By applying the Stein method developed in \cite{FSX19}, we demonstrate that $\mu_\eta$ converges to $\mu$ as $\eta \to 0$ in the $\mathcal{W}_{\mathbf{d}}$-Wasserstein metric. Moreover, we provide an explicit rate of convergence.
\end{itemize}

\subsection{Related work and the novelty}\label{Jel0}
In \cite{CDRX22}, the authors focus on the additive noise and subcritical case, i.e., $\alpha \in (1,2)$. Besides, they also impose the following stronger assumptions:
$$
\|\nabla b\|_\infty + \|\nabla^2 b\|_\infty \leq \kappa_0,
$$
and a monotonicity condition: for some $\kappa_1, \kappa_2 > 0$,
\begin{align}\label{jb0}
\langle x - y, b(x) - b(y) \rangle \leq -\kappa_1 |x - y|^2 + \kappa_2.
\end{align}
Using the Lindeberg principle, the authors in \cite{CDRX22} derive the convergence rate of $\mu_{\eta}$ under the Wasserstein-1 distance. A notable feature of \cite{CDRX22} 
is the use of Malliavin calculus to establish $C^2$-regularity of the semigroup on ${\rm Lip}({\bf d})$ for SDE \eqref{sde:t}.

In this paper, we extend the analysis to the supercritical case, i.e., $\alpha \in \left(\frac{1}{2}, 2\right)$, assuming $b$ and $\sigma$ satisfy $(\mathbf{H}^\gamma)$ and \eqref{Dis}. Leveraging Stein's method and the heat kernel estimates from \cite{MZ22}, we study the convergence rate of $\mu_{\eta}$ under the $\mathcal{W}_{\bf d}$-distance, where ${\bf d}(x,y):=|x-y|^\gamma \wedge 1$ (see Theorem \ref{main1} above). Furthermore, under $(\mathbf{H}^1)$ and \eqref{jb0}, we establish the convergence rate of $\mu_{\eta}$ in the Wasserstein-1 distance (see Theorem \ref{main} above). Compared with \cite{CDRX22}, we make weaker assumptions on $b$ and $\sigma$ and achieve the convergence rate of $\eta^{\frac{1}{\alpha}}$ using alternative techniques.

Although the Malliavin calculus provides powerful tools for analyzing the regularity of the semigroup for SDE \eqref{sde:t}, it has two main limitations in our context: (i) Malliavin calculus typically requires both $b$ and $\sigma$ to possess sufficient regularity to achieve the desired semigroup regularity. Specifically, it is unsuitable for analyzing fractional regularity when $b$ and $\sigma$ satisfy only certain H\"older regularity conditions. (ii) For multiplicative noise, the adaptive Malliavin calculus based on time-change techniques, as used in \cite{CDRX22}, is not applicable because the jump of $X_t$ at time $t$ depends on both $X_{t-}$ and the jump of $L^{(\alpha)}_t$.

In our preliminary study, using the finite-jump approximation method from \cite{WXZ15}, we obtained $C^2$-regularity for the semigroup on ${\rm Lip}({\bf d})$ for SDE \eqref{sde:t} with $\alpha \in (1,2)$. However, this method involves highly complex calculations and estimates and cannot be used in the supercritical case, i.e., $\alpha \in (0,1]$, as higher-order directional derivatives of the SDE solution \eqref{sde:t} only admit $m$-order moment estimates with $m \in (0,\alpha)$. In this paper, in Lemmas \ref{9261} and \ref{le:d2Pt}, we employ heat kernel estimates from \cite{MZ22} and commutator estimates to establish weighted $s$-H\"older ($s > \alpha$) regularity for the semigroup on ${\rm Lip}(\bf d)$ for SDE \eqref{sde:t} with $\alpha \in (0,2)$, which is crucial for deriving the convergence rates for $\mu_{\eta}$ under the $\mathcal{W}_{\bf d}$-distance, where ${\bf d}$ is the bounded H\"older distance.

On the other hand, the Stein/Poisson equation is foundational in mathematical physics with a long research history. Gilbarg and Trudinger \cite{GT83} established interior Schauder estimates for solutions of $\Delta f = g$. Stein \cite{S70} used the Riesz potential to establish global regularity estimates for solutions to $\Delta^{\frac{\alpha}{2}} f = g$ (see also Bass \cite{B09}). Ros-Oton and Serra \cite{RS16} investigated the interior H\"older regularity of solutions to $A f = g$ for certain symmetric $\alpha$-stable operators $A$. Recently, K\"uhn \cite{K21,K22} established global and interior Schauder estimates for solutions to $A f = g$ in certain H\"older-Zygmund spaces, where $A$ is defined by a general Feller process. Notably, many studies have addressed the Stein/Poisson equation defined by infinitesimal generators of specific ergodic diffusion processes. For example, Pardoux and Veretennikov \cite{PV01,PV03,PV05} studied the well-posedness and regularity estimates of the Stein/Poisson equation defined by an elliptic operator associated with an ergodic diffusion process, applying these results to analyze diffusion approximations in multiscale systems.

However, there appear to be few studies addressing well-posedness and regularity estimates of the Stein/Poisson equation for nonlocal operators with unbounded coefficients. Furthermore, as noted in \cite{RX21}, the results of Pardoux and Veretennikov have been widely applied in fields such as the central limit theorem, moderate and large deviations, spectral methods for multiscale systems, the averaging principle, homogenization, and numerical approximations of invariant measures for SDEs and SPDEs. We believe that the results on the Stein/Poisson equation \eqref{Ste} in this paper may be applied to similar topics in a nonlocal context. In future work, we plan to explore these topics further.

In this paper, we investigate the well-posedness and regularity estimates for the Stein/Poisson equation \eqref{Ste} in weighted H\"older spaces. Specifically, for any $g \in {\rm Lip}({\bf d})$, we define
\begin{align}\label{jsj0}
f(x) := \int^\infty_0 P_t \bar g(x) \, \mathrm{d}t,\ \ \bar{g} := g - \mu(g).
\end{align}
We show that the above $f$ belongs to certain weighted H\"older spaces  (see Theorem \ref{Th25}), and indeed is the solution of Stein/Poisson equation \eqref{Ste} (see Theorem \ref{Th216}). 
In particular, we consider two cases:

({\it Case 1.}) Under $(\mathbf{H}^1)$ and \eqref{jb0}, we take ${\bf d}(x, y) = |x - y|$.

({\it Case 2.}) Under $(\mathbf{H}^\gamma)$ and \eqref{Dis}, we take ${\bf d}(x, y) = |x - y|^\gamma \wedge 1$.

In {\it Case 1}, it follows from \cite{Wang20} that there exist constants $C, \lambda > 0$ such that for any $\varphi \in {\rm Lip}({\bf d})$,
\begin{align*}
\|\nabla P_t \varphi\|_\infty \lesssim_C e^{-\lambda t} \|\nabla \varphi\|_\infty, \quad t > 0.
\end{align*}
Using this result and certain commutator estimates, we show that for any $g \in {\rm Lip}({\bf d})$,
the function $f$ defined in \eqref{jsj0} satisfies
\begin{align*}
\|\nabla f\|_\infty + \|\nabla^2 f\|_\infty \lesssim_C \|\nabla g\|_\infty.
\end{align*}
See Lemma \ref{le:n2Pt} and Theorem \ref{Le23} for further details.

In {\it Case 2}, it follows from \cite[Theorem 1.2]{ZZ23} that for any $m \in (0, \alpha)$, there exist constants $C, \lambda > 0$ such that
\begin{align}\label{jsb1}
\sup_{\|\varphi / \rho_m\|_\infty \leq 1} |P_t \varphi(x) - \mu(\varphi)| \lesssim_C \e^{-\lambda t} \rho_m(x), \quad t > 0, \, x \in \mathbb{R}^d,
\end{align}
where $\rho_m(x) := (1 + |x|^2)^{m / 2}$. From \eqref{jsb1}, we obtain only that $f$, defined as in \eqref{jsj0}, and its derivatives exhibit polynomial growth, which complicates the analysis of the convergence rate of $\mathcal{W}_{\bf d}(\mu, \mu_\eta)$. This difficulty arises because, in general, SDE \eqref{sde:t} has only $m$-order moment estimates for $m \in (0, \alpha)$.
To address this, we introduce the cut-off distance ${\bf d}(x, y) = |x - y|^\gamma \wedge 1$. This allows us to show that $f$, defined as in \eqref{jsj0}, and its derivatives exhibit only logarithmic growth (see Theorem \ref{Th25} below). The key observation is the integral
\begin{align*}
\int^\infty_0 1 \wedge (a \e^{-\lambda t}) \, \mathrm{d}t = \lambda^{-1} (1 + \ln a), \quad \forall \lambda > 0, \, a \geq 1.
\end{align*}

\subsection{Organization and notations}

This paper is organized as follows: Section \ref{Solvability} discusses the solvability and regularity of the Stein/Poisson equation \eqref{Ste}, while Section \ref{proof} presents detailed proofs of Theorem \ref{main} and Theorem \ref{main1}. Appendix \ref{EMergo} includes supplementary lemmas and the proof of the ergodicity of the EM scheme presented in \eqref{EM1-1}. Appendix \ref{A:OU} presents the sharp convergence rate of EM schemes for the Ornstein-Uhlenbeck (OU) process.

\bigskip

{\bf Notations.}
Throughout this paper, we use the following conventions: The letter $C$ with or without subscripts denotes an unimportant constant, whose value may be different from line to line. We use $A\asymp B$ and $A\lesssim B$ (or $B\gtrsim A$) to denote $C^{-1} B\leq A\leq C  B$ and $A\leq C  B$ respectively for some unimportant constant $C \geq 1$. For $s,r\geq 0$, denote by $s\wedge r:=\min\{s,r\}$ and $s\vee r:=\max\{s,r\}$.

\section{Solvability of nonlocal Stein/Poisson equations}\label{Solvability}

Throughout this section, we fix $\alpha \in \left(\frac{1}{2}, 2\right)$. Our primary objective is to investigate the well-posedness and regularity properties of the Stein/Poisson equation \eqref{Ste} in weighted H\"older spaces. We consider two distinct cases: \textbf{monotonicity drift} and \textbf{dissipativity drift}. For the monotonicity drift case, we restrict to $\alpha \in (1, 2)$ and derive sharper estimates. For the dissipativity drift case, we explore the full range $\alpha \in \left(\frac{1}{2}, 2\right)$, yielding slightly rougher estimates.

\subsection{Weighted H\"older spaces}
In this section, we recall the definition of weighted H\"older spaces and their equivalent characterizations, which will be used throughout this paper (see \cite{HZZZ24,T06}).

\begin{definition}\label{Def31}
A $C^\infty$-function
$\rho:\mathbb{R}^d \to (0,\infty)$ is called an admissible weight if, for some $C,\ell > 0$,
$$
\rho(x)\lesssim_C \rho(y)(1 + |x - y|^\ell), \quad \forall x, y \in \mathbb{R}^d,
$$
and for each $j \in \mathbb{N}$, there exists a constant $C_j > 0$ such that
$$
|\nabla^j \rho(x)| \leq C_j \rho(x), \quad \forall x \in \mathbb{R}^d.
$$
\end{definition}

\begin{remark}\rm
Typical examples of admissible weights include:
$$
\rho(x) = (\log(\e + |x|^2))^m, \quad \rho(x) = (1 + |x|^2)^{m/2}, \quad m \in \mathbb{R}.
$$
\end{remark}

From now on, we fix an admissible weight $\rho$.
Let $L^\infty_\rho:=L^\infty_\rho(\mathbb{R}^d) $
denote the space of
all measurable functions
on $\mathbb{R}^d$ with finite norm:
$$
\|f\|_{L^\infty_\rho} := \|f \rho\|_\infty := \sup_{x}|f(x) \rho(x)| < \infty.
$$
For $\beta > 0$, let $\bC^\beta_\rho(\mathbb{R}^d)$ denote the weighted H\"older space defined by
$$
\|f\|_{\bC^\beta_\rho} := \sum_{k = 0}^{[\beta]} \|\nabla^k f\|_{L^\infty_\rho} + \sup_{|v| \leq 1} \frac{\|\delta_v \nabla^{[\beta]} f\|_{L^\infty_\rho}}{|v|^{\beta - [\beta]}} < \infty,
$$
where $[\beta]$ denotes the integer part of $\beta$, $\nabla^k$ is the $k$-th order gradient operator, and for a vector-valued function $F:\mathbb{R}^d \to \mathbb{R}^m$,
$$
\delta_v F(x) := F(x + v) - F(x).
$$
By definition, it is clear that for $0 \leq \beta_1 \leq \beta_2$ and two admissible weights $\rho_1, \rho_2$ with $\rho_1 \leq \rho_2$,
$$
\|f\|_{\bC^{\beta_1}_{\rho_1}} \lesssim_C \|f\|_{\bC^{\beta_2}_{\rho_2}}.
$$
Note that for two functions $f, g: \mathbb{R}^d \to \mathbb{R}$,
\begin{align}\label{DS1}
\delta_v (fg)(x) = f(x)\delta_v g(x) + \delta_v f(x) g(x + v).
\end{align}
If $\rho \equiv 1$, then $\bC^\beta_\rho(\mathbb{R}^d)$ becomes the classical H\"older space, and we simply write:
$$
\bC^\beta(\mathbb{R}^d) = \bC^\beta_1(\mathbb{R}^d), \quad \|f\|_{\bC^\beta} := \|f\|_{\bC^\beta_1}.
$$
If $\rho(x) = (1 + |x|^2)^{-\frac{m}{2}}$ for some $m \geq 0$, then we write
$$
L^\infty_m:=L^\infty_\rho,\ \ \mathcal{C}^\beta_m(\mathbb{R}^d) = \bC^\beta_\rho(\mathbb{R}^d), \quad \|f\|_{\mathcal{C}^\beta_m} := \|f\|_{\bC^\beta_\rho},\ \ \beta>0.
$$

Next, we recall an equivalent characterization for $\bC^\beta_\rho(\mathbb{R}^d)$ when $\beta \in (0, \infty) \setminus \mathbb{N}$ and some basic properties of $\bC^\beta_\rho(\mathbb{R}^d)$ (cf. \cite[Theorem 2.6, Remark 2.8, and Corollary 2.9]{HZZZ24}).

\begin{theorem}\label{Le:Cs}
Let $\rho$ be an admissible weight, and let $\beta' \geq \beta > 0$ be non-integers and $\mN\ni m > \beta$.
\begin{enumerate}[(i)]
\item Let $\delta^m_v f := \delta_v(\delta^{m-1}_v f)$ be the $m$-order difference operator. Then we have
\begin{align*}
\|f\|_{\bC^\beta_\rho} \asymp_C \|f\|_{L^\infty_\rho} + \sup_{|v| \leq 1}  \|\delta^m_v f\|_{L^\infty_\rho}/|v|^\beta.
\end{align*}
\item Let $\theta \in (0,1)$ be such that $\theta \beta$ is not an integer. Then we have
\begin{align*}
\|f\|_{\bC^{\theta \beta}_\rho} \lesssim_C \|f\|_{L^\infty_\rho}^{1 - \theta} \|f\|_{\bC^{\beta}_{\rho}}^{\theta}.
\end{align*}
\item
Let $\ell$ be the parameter associated with $\rho$ as in Definition \ref{Def31}
and let $k > \beta' - \beta\geq 0$. Then we have
$$
\|\delta^k_v f\|_{\bC^\beta_\rho} \lesssim_C |v|^{\beta' - \beta} (1 + |v|^\ell) \|f\|_{\bC^{\beta'}_\rho}, \quad \forall v \in \mathbb{R}^d.
$$
\end{enumerate}
\end{theorem}

\begin{remark}\rm
For $\beta \in (0, 2) \setminus \{1\}$, an equivalent norm of $\bC^\beta_\rho$ is given by
$$
\|f\|_{\bC^\beta_\rho} \asymp_C \|f\|_{L^\infty_\rho} + \sup_{|v| \leq 1} \|\delta^{(2)}_v f\|_{L^\infty_\rho}/|v|^\beta,
$$
where $\delta^{(2)}_v$ is the symmetric difference operator defined by
$$
\delta^{(2)}_v f(x) := f(x + v) + f(x - v) - 2 f(x).
$$
\end{remark}

The following lemma provides an alternative expression for $\mathscr{L}^{(\alpha)}_{\sigma(x)} f(v)$.

\begin{lemma}\label{Le33}
Suppose that for some $\gamma \in [0,1]$ and $\kappa_0 > 1$,
\begin{align}\label{Sig}
\kappa_0^{-1} |\xi| \leq |\sigma(x) \xi| \leq \kappa_0 |\xi|, \quad \|\sigma(x) - \sigma(y)\| \leq \kappa_0 |x - y|^\gamma.
\end{align}
Let $\jmath(x,z) := (|z| / |\sigma^{-1}(x) z|)^{d+\alpha} / \det(\sigma(x))$. Then, for any $x, v \in \mathbb{R}^{d}$,
\begin{align}\label{op:La}
\mathscr{L}^{(\alpha)}_{\sigma(x)} f(v)
= \frac{1}{2} \int_{\mathbb{R}^d} \delta^{(2)}_z f(v) \frac{\jmath(x, z)}{|z|^{d+\alpha}} \, \mathrm{d} z,
\end{align}
and there exists a constant $\widetilde\kappa_0 > 1$ such that for all $x, y, z \in \mathbb{R}^d$,
\begin{align}\label{Ho1}
\widetilde\kappa_0^{-1} \leq \jmath(x, z) \leq \widetilde\kappa_0, \quad |\jmath(x, z) - \jmath(y, z)| \leq \widetilde\kappa_0 |x - y|^\gamma.
\end{align}
\end{lemma}

\begin{proof}
Since for each $\epsilon \in (0, 1)$,
$$
\int_{\{\epsilon < |y| \leq 1\}} \sigma(x) y \cdot \nabla f(v) \frac{\mathrm{d} y}{|y|^{d+\alpha}} = 0,
$$
and by symmetry,
$$
\int_{\{\epsilon < |y| < \infty\}} [f(v + \sigma(x) y) - f(v)] \frac{\mathrm{d} y}{|y|^{d+\alpha}} = \int_{\{\epsilon < |y| < \infty\}} [f(v - \sigma(x) y) - f(v)] \frac{\mathrm{d} y}{|y|^{d+\alpha}},
$$
we have
\begin{align}\label{L:sy}
\mathscr{L}^{(\alpha)}_{\sigma(x)} f(v) = \frac{1}{2} \int_{\mathbb{R}^d} \delta^{(2)}_{\sigma(x) y} f(v) \frac{\mathrm{d} y}{|y|^{d+\alpha}},
\end{align}
which gives the desired \eqref{op:La} by the change of variable $z = \sigma(x) y$.
Since $\det(\sigma(x))$ is a polynomial function of $\sigma_{ij}(x)$ and $\sigma^{-1}(x) = \sigma^*(x) / \det(\sigma(x))$, where $\sigma^*$ is the adjoint matrix of $\sigma$, 
estimate \eqref{Ho1} follows from the chain rule and the assumptions.
\end{proof}

We have the following important estimates for $(\mathscr{L}^{(\alpha)}_{\sigma} f)(x) = \mathscr{L}^{(\alpha)}_{\sigma(x)} f(x)$.

\begin{lemma}
Let $\rho$ be an admissible weight satisfying $\rho(x) \lesssim \rho(y)(1 + |x - y|^\ell)$ with $\ell \in [0, \alpha \wedge 1)$. Under \eqref{Sig},
for any $\beta > \alpha$ and $\theta \in [0, 1]$, there exists a constant
$C = C(d, \beta, \alpha, \theta, \rho, \ell, \kappa_0) > 0$
such that for all $f \in \bC^{\theta + \beta}_\rho(\mathbb{R}^d)$,
\begin{align}\label{DS4}
\|\mathscr{L}^{(\alpha)}_{\sigma} f\|_{\bC^{\theta}_\rho} \lesssim_{C} \|f\|_{\bC^{\theta + \beta}_\rho},
\end{align}
and for $\gamma$ being the same as in \eqref{Sig} and for all $x_1, x_2 \in \mathbb{R}^d$,
\begin{align}\label{DS5}
\|\mathscr{L}^{(\alpha)}_{\sigma(x_1)} f - \mathscr{L}^{(\alpha)}_{{\sigma(x_2)}} f\|_{\bC^\theta_\rho} \lesssim_C \|f\|_{\bC^{\theta + \beta}_\rho} |x_1 - x_2|^\gamma.
\end{align}
\end{lemma}

\begin{proof}
Without loss of generality, we assume $\beta \in (\alpha, 2)$.
By Lemma \ref{Le33} and \eqref{DS1}, we have
\begin{align*}
|\delta_v \mathscr{L}^{(\alpha)}_\sigma f(x)|
&\leq \frac{1}{2} \int_{\mathbb{R}^d} \left( |\delta_v \delta^{(2)}_z f(x)| \cdot |\jmath(x + v, z)|
+ |\delta^{(2)}_z f(x)| \cdot |\delta_v \jmath(\cdot, z)(x)| \right) \frac{\mathrm{d} z}{|z|^{d + \alpha}} \\
&\overset{\eqref{Ho1}}{\leq} \frac{\widetilde\kappa_0}{2} \int_{\mathbb{R}^d} \left( |\delta_v \delta^{(2)}_z f(x)| + |v|^\gamma |\delta^{(2)}_z f(x)| \right) \frac{\mathrm{d} z}{|z|^{d + \alpha}}.
\end{align*}
For any $\theta \in [0, 1]$,
by (ii) of Theorem \ref{Le:Cs},
one can choose $\beta_1 \in (\alpha, \beta]$ such that for any $|v| \leq 1$ and $z \in \mathbb{R}^d$,
\begin{align}\label{DS6}
\|\delta_v \delta^{(2)}_z f\|_{L^\infty_\rho} \lesssim |v|^{\theta} \|\delta^{(2)}_z f\|_{\bC^{\theta}_\rho}
\lesssim |v|^\theta (1 + |z|^\ell) \left( (|z|^{\beta_1} \|f\|_{\bC^{\theta + \beta}_\rho}) \wedge \|f\|_{\bC^{\theta}_\rho} \right)
\end{align}
and
$$
\|\delta^{(2)}_z f\|_{L^\infty_\rho} \lesssim (1 + |z|^\ell) \left( (|z|^{\beta_1} \|f\|_{\bC^{\beta}_\rho}) \wedge \|f\|_{L^\infty_\rho} \right).
$$
Hence, for any $\theta \in [0, 1]$, there exists a $\beta_1 \in (\alpha, \beta]$ such that
\begin{align}\label{DS2}
\|\delta_v \mathscr{L}^{(\alpha)}_\sigma f\|_{L^\infty_\rho}
\lesssim |v|^\theta \|f\|_{\bC^{\theta + \beta}_\rho} \int_{\mathbb{R}^d} (1 + |z|^\ell)(|z|^{\beta_1} \wedge 1) \frac{\mathrm{d} z}{|z|^{d + \alpha}}
\lesssim |v|^\theta \|f\|_{\bC^{\theta + \beta}_\rho},
\end{align}
and
\begin{align}\label{DS3}
\|\mathscr{L}^{(\alpha)}_\sigma f\|_{L^\infty_\rho}
\lesssim \|f\|_{\bC^{\beta}_\rho} \int_{\mathbb{R}^d} (1 + |z|^\ell)(|z|^{\beta_1} \wedge 1) \frac{\mathrm{d} z}{|z|^{d + \alpha}}
\lesssim \|f\|_{\bC^{\beta}_\rho}.
\end{align}
Combining \eqref{DS2} and \eqref{DS3}, we obtain \eqref{DS4}.

On the other hand, by Lemma \ref{Le33}, we have
\begin{align*}
|\delta_v(\mathscr{L}^{(\alpha)}_{\sigma(x_1)} f - \mathscr{L}^{(\alpha)}_{\sigma(x_2)}f)(x)|
\leq \int_{\mathbb{R}^d} |\delta_v \delta^{(2)}_z f(x)| \cdot |\jmath(x_1, z) - \jmath(x_2, z)| \frac{\mathrm{d} z}{|z|^{d + \alpha}}.
\end{align*}
Let $\theta \in [0, 1]$ and $\beta_1$ be as in \eqref{DS6}. By \eqref{DS6} and \eqref{Ho1}, we have
$$
\|\delta_v(\mathscr{L}^{(\alpha)}_{\sigma(x_1)} f - \mathscr{L}^{(\alpha)}_{\sigma(x_2)} f)\|_{L^\infty_\rho}
\leq |v|^\theta |x_1 - x_2|^\gamma \|f\|_{\bC^{\theta + \beta}_\rho}
\int_{\mathbb{R}^d} (1 + |z|^\ell)(|z|^{\beta_1} \wedge 1) \frac{\mathrm{d} z}{|z|^{d + \alpha}}
$$
and
$$
\|\mathscr{L}^{(\alpha)}_{\sigma(x_1)} f - \mathscr{L}^{(\alpha)}_{\sigma(x_2)}f\|_{L^\infty_\rho}
\leq  |x_1 - x_2|^\gamma \|f\|_{\bC^{\beta}_\rho}
\int_{\mathbb{R}^d} (1 + |z|^\ell)(|z|^{\beta_1} \wedge 1) \frac{\mathrm{d} z}{|z|^{d + \alpha}}.
$$
Thus we obtain \eqref{DS5}.
\end{proof}

\subsection{Subcritical case: Monotonicity drifts}

In the following, we restrict ourself to $\alpha\in(1,2)$ and assume that $b$ and $\sigma$ satisfy
{\bf (H$^1$)} and for some $\kappa_1,\kappa_2>0$ and all $x,y\in\mR^d$,
\begin{align}\label{mon}
\<x-y,b(x)-b(y)\>\leq-\kappa_1|x-y|^2+\kappa_2.
\end{align}
Clearly, \eqref{mon} implies \eqref{Dis} and there is a unique invariant probability measure $\mu$ associated with $P_t$ so that
$$
\int_{\mR^d}|x|\mu(\dif x)<\infty.
$$
Moreover, under the above assumptions,
 by \cite{Wang20}, there
exist constants $C,\lambda>0$ depending on $\Theta$ (see \eqref{The}) such that
for any Lipschitz function $\varphi$,
\begin{align}\label{11}
\|\nabla P_t\varphi\|_\infty \lesssim_C\e^{-\lambda t}\|\nabla \varphi\|_\infty.
\end{align}

\bl\label{le:n2Pt}
Under $(\mathbf{H}^1)$ and \eqref{mon}, there exists a constant $C=C(\Theta)>0$ such that for any Lipschitz function $\varphi$,
\begin{align}\label{IN:N2Pt}
\|\nabla^2 P_t\varphi\|_\infty\lesssim_C t^{-1/\alpha}\|\nabla\varphi\|_\infty,\ \ \forall t\in(0,1].
\end{align}
\el
\begin{proof}
Without loss of generality, we assume $\varphi\in \bC^2$ and denote by $[A,B]:=AB-BA$ the commutator of operators $A,B$.
Noting that by \eqref{semi} and \eqref{Duh},
$$
 \nabla P_t\varphi-P_t\nabla \varphi
 =\int^t_0\p_s(P_{t-s}\nabla P_s\varphi)\dif s=\int^t_0P_{t-s}[\nabla,\sL]P_s\varphi\dif s,
$$
we have
\begin{align}\label{n2Pt}
\nabla^2 P_t\varphi=\nabla P_t\nabla\varphi+\int^t_0\nabla P_{t-s}[ \nabla,\sL]P_s\varphi\dif s.
\end{align}
Moreover,  we clearly have
$$
[\nabla, b\cdot\nabla] P_s\varphi(x)=\nabla b\cdot\nabla  P_s\varphi(x)
$$
and by Lemma \ref{Le33},
$$
[\nabla, \sL^{(\alpha)}_\sigma]  P_s\varphi(x)
=\frac12\int_{\mR^d}\delta^{(2)}_z P_s\varphi(x)\frac{\nabla\jmath(\cdot,z)(x)\dif z}{|z|^{d+\alpha}}.
$$
By \eqref{Def1}, {\bf (H$^1$)} and \eqref{11}, we have
\begin{align*}
\|[\nabla, \sL] P_s\varphi\|_\infty
&\leq \|[\nabla, b\cdot\nabla] P_s\varphi\|_\infty+\|[\nabla, \sL^{(\alpha)}_\sigma]  P_s\varphi\|_\infty\\
&\lesssim \|\nabla  P_s\varphi\|_\infty
+\int_{\mR^d} (\|\nabla  P_s\varphi\|_\infty|z|)\wedge (\|\nabla^2  P_s\varphi\|_\infty|z|^2)\frac{\dif z}{|z|^{d+\alpha}}\\
&\lesssim\|\nabla \varphi\|_\infty+\|\nabla^2  P_s\varphi\|_\infty.
\end{align*}
Now by \eqref{n2Pt}, we have for $t\in(0,1]$,
\begin{align*}
\|\nabla^2 P_t\varphi\|_\infty
&\leq\|\nabla P_t\nabla\varphi\|_\infty+\int^t_0\|\nabla P_{t-s}[ \nabla,\sL]P_s\varphi\|_\infty\dif s\\
&\overset{\eqref{Grad01}}\lesssim t^{-1/\alpha}\|\nabla\varphi\|_\infty+\int^t_0(t-s)^{-1/\alpha}\|[ \nabla,\sL]P_s\varphi\|_\infty\dif s\\
&\lesssim t^{-1/\alpha}\|\nabla\varphi\|_\infty+\int^t_0(t-s)^{-1/\alpha}(\|\nabla\varphi\|_\infty+\|\nabla^2 P_s\varphi\|_\infty)\dif s.
\end{align*}
Since $\alpha\in(1,2)$, by Gronwall-Volterra's inequality (see \cite{Z10}), we get \eqref{IN:N2Pt}.
\end{proof}

Now we can show the following main result.
\begin{theorem}\label{Le23}
Under $(\mathbf{H}^1)$ and \eqref{mon}, there is a constant $C=C(\Theta)>0$ such that
for any Lipschitz function $g$ and $f=\int^\infty_0 P_t\bar g\dif t$ with $\bar g:=g-\mu(g)$,
\begin{align}\label{92-8}
\|\nabla f\|_\infty+\|\nabla^2 f\|_\infty+\|\sL^{(\alpha)}_\sigma f\|_{\bC^{2-\alpha}}
\lesssim_C\|\nabla g\|_\infty.
\end{align}
Moreover,  for all $x_1,x_2\in\mR^d$,
\begin{align}\label{92-7}
\|\sL^{(\alpha)}_{\sigma(x_1)} f-\sL^{(\alpha)}_{\sigma(x_2)} f\|_\infty
\lesssim_C\|\nabla g\|_\infty|x_1-x_2|.
\end{align}
\end{theorem}
\begin{proof}
By \eqref{11}, it is easy to see that
\begin{align}\label{In:Nf}
\|\nabla f\|_\infty\leq\int^\infty_0\|\nabla P_t\bar g\|_\infty\dif t\lesssim\int^\infty_0\e^{-\lambda t}\|\nabla \bar g\|_\infty\dif t
\lesssim \|\nabla g\|_\infty.
\end{align}
Noting that for $t>1$,
$$
\|\nabla^2P_t\bar g\|_\infty=\|\nabla^2P_1P_{t-1}\bar g\|_\infty\overset{\eqref{IN:N2Pt}}\lesssim \|\nabla P_{t-1}\bar g\|_\infty\overset{\eqref{11}}\lesssim \e^{-\lambda t}\|\nabla g\|_\infty,
$$
we have by Lemma \ref{le:n2Pt},
\begin{align}
\|\nabla^2 f\|_\infty&\leq \int^1_0\|\nabla^2 P_t\bar g\|_\infty\dif t+\int^\infty_1\|\nabla^2P_t\bar g\|_\infty\dif t\no\\
&\lesssim \int^1_0t^{-1/\alpha}\|\nabla\bar g\|_\infty\dif t+\int^\infty_1\e^{-\lambda t}\|\nabla\bar g\|_\infty\dif t\lesssim \|\nabla g\|_\infty.\label{N2fNg}
\end{align}
On the other hand,  for $v\in\mR^d$, by \eqref{op:La} we have
\begin{align*}
\delta_v\sL^{(\alpha)}_\sigma f(x)
=\frac12\int_{\mR^d}\delta_v
\Big(\delta^{(2)}_yf\jmath(\cdot,y)\Big)(x)
\frac{\dif y}{|y|^{d+\alpha}}.
\end{align*}
Since $\|\jmath(\cdot,y)\|_\infty+\|\nabla\jmath(\cdot,y)\|_\infty\leq\tilde\kappa_0$ by Lemma \ref{Le33} with $\gamma=1$, we have
\begin{align*}
\|\delta_v(\delta^{(2)}_yf\jmath(\cdot,y))\|_\infty
&\leq\|\delta_v\delta^{(2)}_yf\|_\infty\|\jmath(\cdot,y)\|_\infty
+\|\delta^{(2)}_yf\|_\infty\|\delta_v\jmath(\cdot,y)\|_\infty\\
&\leq\widetilde\kappa_0\Big(\|\delta_v\delta^{(2)}_yf\|_\infty+\|\delta^{(2)}_yf\|_\infty|v|\Big).
\end{align*}
By \eqref{In:Nf} and \eqref{N2fNg}, we have
$$
\|\delta_v\delta^{(2)}_yf\|_\infty\leq \|\nabla^2 f\|_\infty(|y|^2\wedge (|v| |y|))
\lesssim \|\nabla g\|_\infty(|y|^2\wedge (|v| |y|))
$$
and
$$
\|\delta^{(2)}_yf\|_\infty\leq (\|\nabla^2 f\|_\infty|y|^2)\wedge (\|\nabla f\|_\infty |y|)
\lesssim \|\nabla g\|_\infty(|y|^2\wedge |y|).
$$
Hence, for $|v|\leq 1$,
\begin{align*}
\|\sL^{(\alpha)}_\sigma f\|_\infty
\lesssim& \|\nabla g\|_\infty
\int_{\mR^d}\frac{(|y|^2\wedge |y|)\dif y}{|y|^{d+\alpha}}
\lesssim \|\nabla g\|_\infty
\end{align*}
and
\begin{align*}
\|\delta_v\sL^{(\alpha)}_\sigma f\|_\infty
\lesssim \|\nabla g\|_\infty
\int_{\mR^d}\frac{(|y|^2\wedge (|v| |y|))\dif y}{|y|^{d+\alpha}}\lesssim \|\nabla g\|_\infty|v|^{2-\alpha}.
\end{align*}
Combining the above two estimates, we obtain
$\|\sL^{(\alpha)}_\sigma f\|_{\bC^{2-\alpha}}\lesssim_C\|\nabla g\|_\infty$.
It is similar to \eqref{DS5} that  \eqref{92-7} follows by \eqref{92-8}.
\end{proof}

\subsection{Dissipativity drift}

In this section, we assume that $b$ and $\sigma$ satisfy $(\mathbf{H}^\gamma)$ with $\gamma\in((1-\alpha)_+, 1]$. In addition, we suppose that there exist constants $\kappa_1,\kappa_2>0$ such that for all $x\in\mR^d$,
\begin{align}\label{926-2}
\<x,b(x)\>\leq-\kappa_1|x|^2+\kappa_2.
\end{align}
For given $m\in(0,\alpha)$,
by \cite[Theorem 1.2]{ZZ23}, there is a unique invariant probability measure $\mu$ of $(P_t)_{t\geq 0}$
with finite $m$-order moments,
and there are constants $\lambda, C>0$ depending on $\Theta$ and $m$ such that
\begin{align}\label{eg:x}
\|P_t\varphi-\mu(\varphi)\|_{L^\infty_m}\lesssim_C\e^{-\lambda t}\|\varphi\|_{L^\infty_m}\lesssim_C\e^{-\lambda t}\|\varphi\|_\infty,\ \ t>0.
\end{align}
From this, we clearly have
\begin{align}\label{C0:b}
\sup_{t>0}\|P_t\varphi\|_{L^\infty_m}\lesssim_C\|\varphi\|_{L^\infty_m}.
\end{align}

By \eqref{eg:x} and \eqref{Grad01}, we have the following regularity estimate of $P_tg$.
\bl\label{9261}
For any $\vartheta\in[0,1]$ and $m\in(0,\alpha)$, there is a constant $C=C(\Theta,\vartheta,m)>0$
such that for all $t>0$ and $x\in\mR^d$,
\begin{align}\label{Gr01}
|\nabla P_t g(x)|\lesssim_C \b1_{t\leq 1}t^{\frac{\vartheta\wedge\alpha-1}{\alpha}}\|g\|_{\bC^\vartheta}
+\b1_{t>1}(1\wedge (\e^{-\lambda t}(1+|x|^m)))\|g\|_\infty.
\end{align}
\el
\begin{proof}
For $t\in(0,1]$, noting that
\begin{align*}
\nabla P_t g(x)= \int_{\mR^d}\nabla_x p_t(x,y) g(y)\dif y= \int_{\mR^d} \nabla_x p_t(x,y)(g(y)-g(\theta_t(x)))\dif y,
\end{align*}
by \eqref{Grad01}, \eqref{TG:ep} and the change of variable, we have
\begin{align}
|\nabla P_t {g}(x)|&\lesssim_C  t^{\frac{\alpha-1}{\alpha}}\|g\|_{\bC^\vartheta}\int_{\mR^d}(t^{\frac{1}{\alpha}}+|\theta_t(x)-y|)^{-d-\alpha} (|y-\theta_t(x)|^{\vartheta}\wedge1)\dif y\no\\
&=t^{\frac{\alpha-1}{\alpha}}\|g\|_{\bC^\vartheta}\int_{\mR^d}(t^{\frac{1}{\alpha}}+|y|)^{-d-\alpha} (|y|^{\vartheta}\wedge1)\dif y\lesssim  t^{\frac{\alpha-1}{\alpha}}\|g\|_{\bC^\vartheta}\int_0^\infty \frac{r^{\vartheta}\wedge1}{(t^{\frac{1}{\alpha}}+r)^{\alpha+1}} \dif r\no\\
&\lesssim   t^{\frac{\alpha-1}{\alpha}}\|g\|_{\bC^\vartheta}
\left[t^{-1-\frac1\alpha}\int_0^{t^{\frac{1}{\alpha}}}r^{\vartheta}\dif r
+\int_{t^{\frac{1}{\alpha}}}^{\infty}\frac{r^\vartheta\wedge 1}{r^{\alpha+1}}\dif r\right]\lesssim t^{\frac{\vartheta\wedge\alpha-1}{\alpha}}\|g\|_{\bC^\vartheta}.\label{CA1}
\end{align}
For $t>1$, noting that
$$
\nabla P_t g=\nabla P_1P_{t-1} g=\nabla P_1P_{t-1} (g-\mu(g)),
$$
by \eqref{eg:x} and \eqref{TG:ep}, we have
\begin{align}
|\nabla P_t g(x)|&\leq  \int_{\mR^d}|\nabla_x p_1(x,y)| |P_{t-1}g(y)-\mu(g)|\dif y\no\\
&\lesssim  \|g\|_\infty\int_{\mR^d} p_1(x,y)
\left(1\wedge (\e^{-\lambda (t-1)}(1+|y|^m))\right)\dif y\no\\
&\lesssim \|g\|_\infty(1\wedge (\e^{-\lambda t}(1+|x|^m))),\label{CA2}
\end{align}
where we use the fact that
$$
 \int_{\mR^d} p_1(x,y)(1+|y|^m)\dif y= 1+\mE |X^x_1|^m\stackrel{\eqref{C0:b}}{\lesssim} 1+|x|^m.
$$
Combining \eqref{CA1} and \eqref{CA2}, we obtain \eqref{Gr01} and complete the proof.
\end{proof}
\br\label{jys1}\rm
For the short-time estimate, in order to make $\nabla P_t g$ integrable around $0$, we need to require
$
\frac{\vartheta \wedge \alpha - 1}{\alpha} > -1,
$
which leads to $\alpha > \frac{1}{2}$. This is also the reason that we cannot tackle the case $\alpha \in \left(0, \frac{1}{2}\right]$.
For the large-time estimate, one cannot take $m = 0$ in \eqref{Gr01}.
This is because under dissipative condition \eqref{926-2} we cannot take $m = 0$ in \eqref{eg:x}.
\er

To show the regularity of the solution to the Stein/Poisson equation,
we need the following estimate.
\bl
For any $\alpha\in(0,2)$, there is a constant $C=C(d,\alpha)>0$ such that
for all $f:\mR^d\to\mR$ with $\|\Delta^{\frac{\alpha}{2}}f\|_\infty<\infty,$
\begin{align}\label{DQ1*}
|\delta^{(2)}_vf(x)|+|\delta^2_vf(x)|\lesssim_C| v|^{\alpha}(\sU_{f,|v|}(x)+\sU_{f,|v|}(x\pm v)),\ \ \forall x,v\in\mR^d,
\end{align}
where $\sU_{f,r}(x\pm v)=\sU_{f,r}(x+v)+\sU_{f,r}(x-v)$, and for $r>0$,
\begin{align*}
\sU_{f,r}(x):=\sum_{n=-\infty}^\infty \frac{2^{-|n|(\alpha\wedge(2-\alpha))}}{(2^{n}r)^d}\int_{|z|\leq 2^{n+1}r}|\Delta^{\frac{\alpha}{2}}f(x-z)|\dif z.
\end{align*}
\el
\begin{proof}
We only consider the second order symmetric difference oeprator $\delta^{(2)}$.
By suitable smooth approximation, we may assume
$f$ is smooth and has compact support.
Note that by Fourier's transform,
$$
f(x)=\Delta^{-\frac\alpha 2}\Delta^{\frac{\alpha}{2}}f(x)=K*\Delta^{\frac{\alpha}{2}}f(x),
$$
where $K(x)=c_{d,\alpha}|x|^{-d+\alpha}$ and the constant  $c_{d,\alpha}>0$ is
chosen
so that $\hat K(\xi)=|\xi|^{-\alpha}$ (see \cite[p. 117]{S70}).
Fix $v\in\R^d$.
From this, we have
\begin{align*}
|\delta^{(2)}_vf(x)|=|\delta^{(2)}_vK*\Delta^{\frac{\alpha}{2}}f(x)|
&\leq\left(\int_{|z|\leq 2|v|}+\int_{|z|>2|v|}\right)|\delta^{(2)}_vK(z)|\cdot|\Delta^{\frac{\alpha}{2}}f(x-z)|\dif z.
\end{align*}
For $|z|>2|v|$, by Taylor's expansion we have
\begin{align*}
|\delta^{(2)}_v K(z)|&\leq |v|^2\int^1_0\int^1_{-1}|\nabla^2 K(z+\theta\theta' v)|\dif\theta\dif\theta'\\
&\lesssim |v|^2\int^1_0\int^1_{-1}|z+\theta\theta' v|^{\alpha-d-2}\dif\theta\dif\theta'\lesssim |v|^2|z|^{\alpha-d-2}.
\end{align*}
Hence,  
\begin{align*}
\int_{|z|> 2|v|}|\delta^{(2)}_vK(z)|\cdot&|\Delta^{\frac{\alpha}{2}}f(x-z)|\dif z
\lesssim |v|^2\int_{|z|> 2|v|}\frac{|\Delta^{\frac{\alpha}{2}}f(x-z)|}{|z|^{d+2-\alpha}}\dif z\\
&= |v|^2\sum_{n=1}^\infty\int_{2^n|v|\leq |z|<2^{n+1}|v|}\frac{|\Delta^{\frac{\alpha}{2}}f(x-z)|}{|z|^{d+2-\alpha}}\dif z\\
&\leq |v|^2\sum_{n=1}^\infty(2^n|v|)^{\alpha-d-2}\int_{|z|<2^{n+1}|v|}|\Delta^{\frac{\alpha}{2}}f(x-z)|\dif z\\
&=|v|^\alpha\sum_{n=1}^\infty\frac{2^{-n(2-\alpha)}}{(2^n|v|)^d}\int_{|z|<2^{n+1}|v|}|\Delta^{\frac{\alpha}{2}}f(x-z)|\dif z
\leq |v|^\alpha\sU_{f,|v|}(x).
\end{align*}
On the other hand, we also have
\begin{align*}
\int_{|z|\leq 2|v|}|K(z)|\cdot&|\Delta^{\frac{\alpha}{2}}f(x-z)|\dif z
=\sum_{n=0}^\infty\int_{2^{-n}|v|<|z|\leq 2^{-n+1}|v|}|K(z)|\cdot|\Delta^{\frac{\alpha}{2}}f(x-z)|\dif z\\
&\leq 2\sum_{n=0}^\infty (2^{-n}|v|)^{\alpha-d}\int_{2^{-n}|v|<|z|\leq 2^{-n+1}|v|}|\Delta^{\frac{\alpha}{2}}f(x-z)|\dif z\\
&\leq 2|v|^\alpha\sum_{n=0}^\infty \frac{2^{-n\alpha}}{(2^{-n}|v|)^d}\int_{|z|\leq 2^{-n+1}|v|}|\Delta^{\frac{\alpha}{2}}f(x-z)|\dif z
\leq 2|v|^\alpha\sU_{f,|v|}(x),
\end{align*}
and similarly,
\begin{align*}
&\int_{|z|\leq 2|v|}|K(z\pm v)|\cdot|\Delta^{\frac{\alpha}{2}}f(x-z)|\dif z
\leq\int_{|z\pm v|\leq 4|v|}|K(z\pm v)|\cdot|\Delta^{\frac{\alpha}{2}}f(x-z)|\dif z\\
&\qquad\qquad=\int_{|z|\leq 4|v|}|K(z)|\cdot|\Delta^{\frac{\alpha}{2}}f(x\pm v-z)|\dif z
\lesssim |v|^\alpha\sU_{f,|v|}(x\pm v).
\end{align*}
Combining the above calculations, we obtain the desired estimate.
\end{proof}

Using the above lemma, we can show the following result, which will be used to establish the large time estimate.

\begin{lemma}\label{Le:33}
Let $m\in[0,\alpha\wedge(2-\alpha))$ and $\beta\in[0,\alpha]$. For any $T>0$, there is a constant $C=C(\Theta,m,\beta, T)>0$ such that for all $x,v\in\mathbb{R}^d$ and $t\in(0,T]$,
\begin{align}\label{LK1}
|\delta^2_v P_tf(x)|+|\delta^{(2)}_v P_tf(x)|\lesssim_C
|v|^\beta t^{-\frac{\beta}{\alpha}} (1+|x|^m+|v|^m)\|f\|_{L^\infty_m}.
\end{align}
\end{lemma}

\begin{proof}
By interpolation,
it suffices to consider $\beta=0$ and $\beta=\alpha$.
For $\beta=0$, by \eqref{C0:b}, \eqref{LK1} clearly holds. Now, we consider the case $\beta=\alpha$.
By definition, we have
$$
|\delta^2_v P_tf(x)|\leq\int_{\mathbb{R}^d}|\delta^2_v p_t(\cdot,y)(x)|\,|f(y)| \, \mathrm{d} y.
$$
Note that by \eqref{DQ1*},
$$
|\delta^2_v p(t,\cdot,y)(x)|\lesssim | v|^{\alpha}
(\sU_{p_t(\cdot,y),|v|}(x)+\sU_{p_t(\cdot,y),|v|}(x\pm v)),
$$
where
$$
\sU_{p_t(\cdot,y),r}(x)=\sum_{n=-\infty}^\infty \frac{2^{-|n|(\alpha\wedge(2-\alpha))}}{(2^{n}r)^d}
\int_{|z|\leq 2^{n+1}r}|\Delta^{\alpha/2} p_t(\cdot,y)(x-z)| \, \mathrm{d} z.
$$
Hence, by Fubini's lemma and \eqref{Grad01}, for any $t\in (0,T]$, we have
\begin{align*}
|\delta^2_v P_tf(x)|
&\lesssim |v|^\alpha t^{-1} \sum_{n=-\infty}^\infty \frac{2^{-|n|(\alpha\wedge(2-\alpha))}}{(2^{n}|v|)^d}
\Bigg[\int_{|z|\leq 2^{n+1}|v|}\int_{\mathbb{R}^d} p_t(x-z,y)|f(y)| \, \mathrm{d} y \, \mathrm{d} z \\
&\qquad\qquad+\int_{|z|\leq 2^{n+1}|v|}\int_{\mathbb{R}^d} p_t(x\pm v-z,y)|f(y)| \, \mathrm{d} y \, \mathrm{d} z \Bigg].
\end{align*}
By the assumption and SDE \eqref{sde:t}, it is standard to derive that for any $m\in[0,\alpha)$ and $t\in(0,T]$,
$$
\int_{\mathbb{R}^d} p_t(x\pm v-z,y)|f(y)| \, \mathrm{d} y=\mathbb{E} |f(X_t^{x\pm v-z})|\lesssim (1+|x\pm v-z|^m)\|f\|_{L^\infty_m}.
$$
Therefore, for any $t\in(0,T]$, it holds that
\begin{align*}
|\delta^2_v P_tf(x)|
&\lesssim \|f\|_{L^\infty_m}
|v|^\alpha t^{-1} \sum_{n=-\infty}^\infty \frac{2^{-|n|(\alpha\wedge(2-\alpha))}}{(2^{n}|v|)^d}
\int_{|z|\leq 2^{n+1}|v|}(1+|x\pm v-z|^m) \, \mathrm{d} z\\
&\lesssim \|f\|_{L^\infty_m}
|v|^\alpha t^{-1} \sum_{n=-\infty}^\infty 2^{-|n|(\alpha\wedge(2-\alpha))}
(1+|x|^m+(2^{n+1}|v|)^m)\\
&\lesssim \|f\|_{L^\infty_m}|v|^\alpha t^{-1}(1+|x|^m+|v|^m),
\end{align*}
where the last step is due to $m\in[0,\alpha\wedge(2-\alpha))$. The proof is complete.
\end{proof}

Now, we prove the following crucial short-time and long-time regularity estimates of the semigroup, which will be used to establish the regularity estimate of the Stein/Poisson equation. The key point here is that the regularity index can be greater than $\alpha$.

\bl\label{le:d2Pt}
Let $\alpha\in(\frac12,2)$. Under $(\mathbf{H}^\gamma)$ and \eqref{926-2},
for any  $\beta\in(0,\alpha)$ and $m\in(0,\alpha\wedge(2-\alpha))$,  there exist
$\tau\in(0,1)$ and constant
$C=C(\Theta,m,\beta)>0$ such that for any $g\in\bC^\gamma$ and $x\in\mR^d$,
\begin{align}\label{Reg0}
\sup_{|v|\leq 1}\frac{|\delta^3_v P_{t}{g}(x)|}{|v|^{\beta+\gamma}}
\lesssim_C
\left\{
\begin{aligned}
& t^{-\frac{\beta}{\alpha}} \|g\|_{\bC^{\gamma}},&\ \ t\in(0,2],\\
&\sI_m^t(x)^\tau\|g\|_\infty,&\ \ t\in (2,\infty),
\end{aligned}
\right.
\end{align}
where $\sI^{t}_m(x):=1\wedge (\e^{-\lambda t}(1+|x|^m))$ and $\lambda$ is from \eqref{eg:x}.
\el
\begin{proof}
Without loss of generality, we assume $g\in \bC^2$ and denote by $[A,B]:=AB-BA$ the commutator of operators $A,B$.
Fix $t_0\geq 0$.
Note that by \eqref{semi} and \eqref{Duh},
\begin{align*}
 \delta_v P_{t+t_0}{g}-P_t\delta_v {P_{t_0}g}&=\int^t_0\p_s(P_{t-s}\delta_v P_{s+t_0}{g})\dif s
=\int^t_0P_{t-s}[\delta_v, \sL]P_{s+t_0}{g}\dif s.
\end{align*}
In particular, for $|v|\leq 1$,
\begin{align}\label{CS1}
 \delta^{3}_v P_{t+t_0}{g}=\delta^2_vP_t\delta_v P_{t_0}{g}+\int^t_0\delta^2_vP_{t-s}[\delta_v, \sL]P_{s+t_0}{g}\dif s.
\end{align}
Note that by definition \eqref{Def1},
\begin{align*}
[\delta_v, \sL]P_{s+t_0}{g}=[\delta_v, b\cdot\nabla] P_{s+t_0}{g}+[\delta_v, \sL^{(\alpha)}_\sigma]P_{s+t_0}{g}.
\end{align*}
By \eqref{DS1} and $(\mathbf{H}^\gamma)$, we have for all $|v|\leq 1$,
\begin{align}\label{CC1}
|[\delta_v, b\cdot\nabla] P_{s+t_0}{g}(x)|=|\delta_v b(x)\cdot \nabla P_{s+t_0}{g}(x+v)|
\leq \kappa_0 |v|^{\gamma} |\nabla P_{s+t_0}{g}(x+v)|,
\end{align}
and by Lemma \ref{Le33}, for each $s>0$ and $x,v\in\mR^d$,
\begin{align}\label{Dv-LP}
|[\delta_v, \sL^{(\alpha)}_\sigma]P_{s+t_0}{g}(x)|
&=\left|\int_{\mR^d}\frac{(\delta^{(2)}_yP_{s+t_0}g)(x+v)\delta_v
\jmath(\cdot,y)(x)}{|y|^{d+\alpha}}\dif y\right|\no\\
&\leq\widetilde\kappa_0 |v|^\gamma
\int_{\mR^d}\frac{|(\delta^{(2)}_yP_{s+t_0}g)(x+v)|}{|y|^{d+\alpha}}\dif y.
\end{align}
In the following, all the constants only depend on $\Theta$ and $\beta$. We  fix $|v|\leq 1$.

\medskip

\noindent{\it (Small time estimate)} We assume $t_0=0$.
In this case, by \eqref{CC1} and \eqref{Gr01}, we have
\begin{align*}
\|[\delta_v, b\cdot\nabla] P_s{g}\|_\infty\lesssim |v|^{\gamma} s^{\frac{\gamma\wedge\alpha-1}{\alpha}}\|g\|_{\bC^\gamma},\quad \forall s\in(0,2],\ |v|\leq 1,
\end{align*}
and for $\beta_0\in(\alpha,2]$, by \eqref{Dv-LP} with $t_0=0$,
\begin{align*}
\|[\delta_v, \sL^{(\alpha)}_\sigma]P_{s}{g}\|_\infty
&\lesssim |v|^\gamma\int_{\mR^d}\Big(\|g\|_\infty\b1_{|y|> 1} +\|P_{s}g\|_{\bC^{\beta_0}}|y|^{\beta_0}\b1_{|y|\leq 1}\Big)\frac{\dif y}{|y|^{d+\alpha}}\no\\
&\lesssim |v|^\gamma\Big(\|g\|_\infty+\|P_{s}g\|_{\bC^{\beta_0}}\Big),\quad \forall s\in(0,1],\ |v|\leq 1.
\end{align*}
Hence, for $\beta_0\in(\alpha,2]$ and all $s\in(0,2]$ and $\ |v|\leq 1$,
$$
\|[\delta_v, \sL]P_{s}{g}\|_\infty\lesssim |v|^\gamma
\Big(s^{\frac{\gamma\wedge\alpha-1}{\alpha}}\|g\|_{\bC^\gamma}
+\|P_{s}g\|_{\bC^{\beta_0}}\Big).
$$
Thus, by \eqref{CS1} with $t_0=0$,
since $\gamma\wedge\alpha+\alpha>1$, we have for any $\beta\in[0,\alpha)$,
\begin{align}\label{es:2h1v}
\|\delta^{3}_v P_{t}{g}&\|_\infty
\leq \|\delta^2_vP_t\delta_v{g}\|_\infty+\int^t_0\|\delta^2_vP_{t-s}[\delta_v, \sL]P_{s}{g}\|_\infty\dif s\no\\
&\overset{\eqref{LK1}}\lesssim |v|^\beta t^{-\frac\beta\alpha}\|\delta_v{g}\|_\infty
+|v|^\beta \int^t_0(t-s)^{-\frac\beta\alpha}\|[\delta_v, \sL]P_{s}{g}\|_\infty\dif s\no\\
&\ \ \lesssim |v|^{\beta+\gamma} t^{-\frac\beta\alpha}\|g\|_{\bC^\gamma}
+|v|^{\beta+\gamma} \int^t_0(t-s)^{-\frac\beta\alpha} \Big(s^{\frac{\gamma\wedge\alpha-1}{\alpha}}\|g\|_{\bC^\gamma}
+\|P_{s}g\|_{\bC^{\beta_0}}\Big)\dif s\no\\
&\ \ \lesssim |v|^{\beta+\gamma} \left(t^{-\frac\beta\alpha}\|g\|_{\bC^\gamma}+
\int^t_0(t-s)^{-\frac\beta\alpha} \|P_{s}g\|_{\bC^{\beta_0}}\dif s\right).
\end{align}
In particular, if one chooses $\beta\in((\alpha-\gamma)_+,\alpha)$ and
let $\beta_0=(\beta+\gamma)\wedge 2$, then by Lemma \ref{Le:Cs},
$$
\|P_tg\|_{\bC^{\beta_0}}\lesssim \|P_tg\|_\infty+\sup_{|v|\leq 1}\|\delta^{3}_v P_{t}{g}\|_\infty/|v|^{\beta_0}
\lesssim t^{-\frac\beta\alpha}\|g\|_{\bC^\gamma}+
\int^t_0(t-s)^{-\frac\beta\alpha}\|P_{s}g\|_{\bC^{\beta_0}}\dif s,
$$
which yields by Gronwall's inequality of Volterra's type (see \cite{Z10}) that for any $t\in(0,2]$,
\begin{align}\label{Sma1}
\|P_tg\|_{\bC^{(\beta+\gamma)\wedge 2}}=\|P_tg\|_{\bC^{\beta_0}}\lesssim t^{-\frac{\beta}{\alpha}} \|g\|_{\bC^{\gamma}}.
\end{align}
Substituting this back into \eqref{es:2h1v}, we get \eqref{Reg0} for $t\in(0,2]$.

\medskip

\noindent {\it (Large time estimate)}
In the following we fix
$$
\beta\in ((\alpha-\gamma)_+,\alpha), \ m\in(0,\alpha\wedge(2-\alpha)), \ t\in(2,\infty).
$$
Define $t_0:=t-1$.
By \eqref{CC1} and \eqref{Gr01}, for $s\in(0,1]$, we have
\begin{align}\label{KF1}
|[\delta_v, b\cdot\nabla]P_{s+t_0}{g}(x)|\lesssim  |v|^{\gamma}\sI^{s+t_0}_m(x) \|g\|_\infty\lesssim  |v|^{\gamma}\sI^{t_0}_m(x) \|g\|_\infty,\ \ |v|\leq 1.
\end{align}
For $s\in(0,1]$ and $|y|\leq 1$, we have
$$
\|\delta^{(2)}_yP_{s+t_0}g\|_\infty
=\|\delta^{(2)}_yP_sP_{t_0}g\|_\infty
\stackrel{\eqref{Sma1}}{\lesssim} s^{-\frac{\beta}{\alpha}}|y|^{(\beta+\gamma)\wedge 2}\|P_{t_0}g\|_{\bC^\gamma}
\stackrel{\eqref{Gr01}}{\lesssim} s^{-\frac{\beta}{\alpha}}|y|^{\beta+\gamma}\|g\|_\infty,
$$
which implies that for any $R\in(0,1)$,
\begin{align}\label{AQ2}
\int_{|y|\leq R}\frac{|(\delta^{(2)}_yP_{s+t_0}g)(x)|}{|y|^{d+\alpha}}\dif y
\lesssim s^{-\frac{\beta}{\alpha}} \|g\|_\infty
\int_{|y|\leq R}|y|^{(\beta+\gamma)\wedge 2-d-\alpha}\dif y
\lesssim s^{-\frac{\beta}{\alpha}} \|g\|_\infty R^{(\beta+\gamma)\wedge 2-\alpha}.
\end{align}
Let $\bar g(x):=g(x)-\mu(g)$. By \eqref{eg:x}, for any $s\in(0,1]$ and $y\in\mR^d$, we have
$$
|(\delta^{(2)}_yP_{s+t_0}g)(x)|=|(\delta^{(2)}_yP_{s+t_0}\bar g)(x)|
\lesssim\e^{-\lambda t_0}\|g\|_\infty(1+|x|^m+|y|^m),
$$
which implies  for any $R\in(0,1)$,
\begin{align*}
\int_{|y|\geq R}\frac{|(\delta^{(2)}_yP_{s+t_0}g)(x)|}{|y|^{d+\alpha}}\dif y
&\lesssim
\e^{-\lambda t_0}\|g\|_\infty\int_{|y|\geq R}\frac{1+|x|^m+|y|^m}{|y|^{d+\alpha}}\dif y\\
&\lesssim\e^{-\lambda t_0}\|g\|_\infty(R^{-\alpha}(1+|x|^m)+R^{m-\alpha})\\
&\lesssim\e^{-\lambda t_0}\|g\|_\infty R^{-\alpha}(1+|x|^m).
\end{align*}
Moreover, we clearly have for $R\in(0,1)$,
$$
\int_{|y|\geq R}\frac{|(\delta^{(2)}_yP_{s+t_0}g)(x)|}{|y|^{d+\alpha}}\dif y
\leq\int_{|y|\geq R}\frac{4\|g\|_\infty}{|y|^{d+\alpha}}\dif y\lesssim\|g\|_\infty R^{-\alpha}.
$$
Combining the above two estimates, we get for $R\in(0,1)$,
\begin{align}\label{AQ1}
\int_{|y|\geq R}\frac{|(\delta^{(2)}_yP_{s+t_0}g)(x)|}{|y|^{d+\alpha}}\dif y
\lesssim \|g\|_\infty R^{-\alpha}(1\wedge(\e^{-\lambda t_0}(1+|x|^m)))
=\|g\|_\infty R^{-\alpha}\sI_m^{t_0}(x).
\end{align}
Hence, by \eqref{Dv-LP}, \eqref{AQ1} and \eqref{AQ2} with
$R=\big(\sI_m^{t_0}(x)\big)^{\ff{1}{(\beta+\gamma)\wedge 2}}<1$
and $\tau=1-\frac{\alpha}{(\beta+\gamma)\wedge 2}$,
\begin{align*}
|[\delta_v, \sL^{(\alpha)}_\sigma]P_{s+t_0}{g}(x)|
\lesssim|v|^\gamma s^{-\frac{\beta}{\alpha}} \|g\|_\infty\Big(
R^{(\beta+\gamma)\wedge 2-\alpha}
+R^{-\alpha}\sI_m^{t_0}(x)\Big)=|v|^\gamma s^{-\frac{\beta}{\alpha}} \|g\|_\infty\big(\sI_m^{t_0}(x)\big)^\tau,
\end{align*}
which together with \eqref{KF1} yields that for all $s\in(0,1]$,
$$
|[\delta_v, \sL]P_{s+t_0}{g}(x)|\lesssim |v|^{\gamma}s^{-\frac{\beta}{\alpha}}
\|g\|_\infty(\sI^{t_0}_m(x))^\tau
=|v|^{\gamma}s^{-\frac{\beta}{\alpha}}
\|g\|_\infty\Big(1\wedge (\e^{-\lambda t}(1+|x|^m))^\tau\Big).
$$
Now, applying \eqref{LK1} with $f=[\delta_v, \sL]P_{s+t_0}{g}$ and $m$ being $0$ and $m\tau$,
we get for $s\in(0,1)$,
\begin{align}\label{FD1}
|\delta^2_vP_{1-s}[\delta_v, \sL]P_{s+t_0}{g}(x)|
\lesssim |v|^{\beta+\gamma} (1-s)^{-\frac{\beta}{\alpha}}s^{-\frac{\beta}{\alpha}} \|g\|_\infty (\sI^{t_0}_m(x))^\tau.
\end{align}
On the other hand, by \eqref{LK1} with $f=\delta_v P_{t_0}{g}$, noting that $t_0>1$, $|v|\leq 1$ and
$$
|\delta_v P_{t_0}{g}(x)|\leq |v|\int^1_0 |\nabla P_{t_0}{g}(x+rv)|\dif r
\lesssim \|g\|_\infty |v|\sI^{t_0}_m(x),
$$
we have
\begin{align}\label{FD2}
|\delta^2_vP_1\delta_v P_{t_0}{g}(x)|
\lesssim |v|^{\alpha+1}\|g\|_\infty \sI^{t_0}_m(x).
\end{align}
Substituting \eqref{FD1} and \eqref{FD2} into \eqref{CS1} with $t$ replaced by $1$, we get
\begin{align*}
|\delta^3_v P_t{g}(x)|&\lesssim |v|^{\alpha+1}\|g\|_\infty \sI^{t_0}_m(x)
+|v|^{\beta+\gamma}\|g\|_\infty(\sI^{t_0}_m(x))^\tau
\int^1_0 (1-s)^{-\frac{\beta}{\alpha}}s^{-\frac{\beta}{\alpha}}\dif s,
\end{align*}
which in turn gives the large time estimate in \eqref{Reg0}.
\end{proof}

In the following we write
$$
\rho_{\rm log}(x):=(\log(\e+|x|^2))^{-1},\ \ \bC^\beta_{\log}(\mR^d):=\bC^\beta_{\rho_{\rm log}}(\mR^d).
$$
Now, we are ready to study the regularity of $f$.

\begin{theorem}\label{Th25}
Let $\alpha\in(\frac12,2)$. Under $(\mathbf{H}^\gamma)$ and \eqref{926-2},
for any $\beta\in(0,\alpha)$,
there is a constant $C=C(\Theta,\beta)>0$ such that for
$f:=\int^\infty_0 P_t\bar g \dif t$ with $\bar g=g-\mu(g)$ and $g\in\bC^\gamma$,
\begin{align}\label{926-4}
\|f\|_{\bC^{1}_{\rm log}}+\|f\|_{\bC^{\beta+\gamma}_{\rm log}}
\lesssim_{C}\|g\|_{\bC^\gamma}.
\end{align}
\end{theorem}
\begin{proof}
First of all, for $m\in(0,\alpha)$, by \eqref{eg:x} we have
\begin{align*}
|f(x)|\leq \int^\infty_0|P_t \bar g(x)| \dif t
\lesssim \|g\|_{\infty}\int^\infty_01\wedge (\e^{-\lambda t}(1+|x|^m))\dif t.
\end{align*}
For $c\geq 1$ and $\lambda>0$, letting $t_0:=\lambda^{-1}\ln c$, we have
\begin{align}\label{key:AL}
\int^\infty_0 1\wedge(c\e^{-\lambda t})\dif t
=t_0+c\int^\infty_{t_0} \e^{-\lambda t}\dif t
=\lambda^{-1}(1+\ln c).
\end{align}
Hence,
\begin{align}\label{AS3}
|f(x)|\lesssim \|g\|_{\infty}/\rho_{\rm log}(x).
\end{align}
Similarly, by Lemma \ref{9261}, we also have
\begin{align}\label{926-3}
|\nabla f(x)|\lesssim \|g\|_{\bC^\gamma}/\rho_{\rm log}(x).
\end{align}
Moreover, by \eqref{Reg0}  and \eqref{key:AL}, we also have
\begin{align*}
|\delta^{3}_vf(x)|
&\leq \int^\infty_0  |\delta^{3}_vP_t \bar{g}(x)|\dif t
=\left(\int^{1}_0+\int^\infty_{1}\right)  |\delta^{3}_vP_t \bar{g}(x)|\dif t\\
&\lesssim|v|^{\beta+\gamma} \|g\|_{\bC^\gamma}
\left[\int^{1}_0 t^{-\frac{\beta}{\alpha}}\dif t
+\int^\infty_{1}1\wedge (\e^{-\lambda t}(1+|x|^m))^\tau\dif t\right]
\lesssim|v|^{\beta+\gamma}\|g\|_{\bC^\gamma}/\rho_{\rm log}(x),
\end{align*}
which, together with \eqref{AS3}, \eqref{926-3} and by Lemma \ref{Le:Cs}, yields the desired estimate.
\end{proof}

\subsection{The Stein/Poisson equation}\label{jsq:sp}
Let $\cC_0$ be the space of all continuous functions that vanishes at infinity.
Let $\sA$ be the generator of $(P_t)_{t\geq 0}$ in Banach space $\cC_0$, i.e.,
$$
\cD(\sA):=\left\{f\in \cC_0: \lim_{t\downarrow 0}\frac{P_tf-f}{t}=:\sA f\mbox{ exists in $\cC_0$}\right\}.
$$
By It\^o's formula, it is easy to see that
$$
\forall f\in \bC^2\cap \cC_0\subset\cD(\sA),\ \ \sA f=\sL f.
$$
Moreover, $(P_t)_{t\geq 0}$ is a $C_0$-semigroup on $\cC_0$. Indeed,
by \cite[Proposition 2.4, p.89]{RY99}, it suffices to show that for any $f\in \cC_0$ and
 each $x\in\mR^d$,
\begin{align}\label{DD1}
\lim_{t\to 0}|P_tf(x)-f(x)|=0,
\end{align}
and for any $t>0$,
\begin{align}\label{DD2}
P_tf\in\cC_0.
\end{align}
For \eqref{DD1}, it follows by the dominated convergence theorem. For \eqref{DD2}, we only need to show
\begin{align}\label{DD3}
\lim_{x\to\infty}P_tf(x)=0.
\end{align}
For given $\eps>0$, since $f\in \cC_0$, there is an $R>0$ such that
$$
\sup_{|x|>R}|f(x)|\leq\eps.
$$
Thus by \eqref{TG:ep} and the dominated convergence theorem, since $\lim_{x\to\infty}\theta_t(x)=\infty$, we have
$$
\lim_{x\to\infty}|P_tf(x)|\leq\eps+\lim_{x\to\infty}\mE (|f(X^x_t)\b1_{|X^x_t|<R})\leq
\eps+C t\lim_{x\to\infty}\int_{|y|<R}\frac{|f(y)|}{(t^{\frac{1}{\alpha}}+|\theta_t(x)-y|)^{d+\alpha}}\dif y\leq\eps,
$$
which gives \eqref{DD3} by the arbitrariness of $\eps$.

We have the following standard result.

\begin{proposition}
\label{rep:ss}
Let $g\in \cC_0$ and $\bar g=g-\mu(g)$. Then it holds that
\begin{align}\label{stein:s1}
f(x):=\int^\infty_0P_t \bar g(x)\dif t\in\cD(\sA)\mbox{ and }\sA f=-\bar g.
\end{align}
\end{proposition}

\begin{proof}
By \eqref{eg:x}, it is clear that for each $x\in\mR^d$ and $m\in(0,\alpha)$,
\begin{align*}
|f(x)|\leq\int_{0}^{\infty}|P_{t}\bar{g}(x)|\dif t\leq C_m\lambda^{-1}(1+|x|^m)<\infty
\end{align*}
Hence, $f$ is well-defined.
For $\eps>0$, let $U_\eps$ be the resolvent of $\sA$, that is, for any $g\in \cC_0$,
$$
U_\eps g:=\int^\infty_0 \e^{-\eps t} P_tg\dif t\in\cD(\sA).
$$
By the Hille-Yosida Theorem (see \cite[Theorem 3.1]{P83}),
$\sA$  is closed and for any $\eps>0$ and $g\in \cC_0$,
$$
 (\eps\mI-\sA)U_\eps g=g\Longleftrightarrow U_\eps g=(\eps\mI-\sA)^{-1}g.
$$
In particular,
\begin{align}\label{Re1}
\sA U_\eps \bar{g}=(\eps U_\eps -\mI)\bar g=:g_\eps.
\end{align}
Noting that
\begin{align*}
\lim_{\eps\downarrow0}U_\eps\bar g=\int_{0}^{\infty}P_{t}\bar{g}\dif t=f,\  \ \lim_{\eps\downarrow0}g_\eps=-\bar g\mbox{ in $\cC_0$},
\end{align*}
we have  $f\in\cD(\sA)$ and $\lim_{\eps\downarrow0}\sA U_\eps\bar g=\sA f$.
By taking limits for \eqref{Re1}, we get \eqref{stein:s1}.
\end{proof}

Now we are ready to
prove
the main result of this section.
\bt\label{Th216}
Suppose that one of the following two conditions holds:

(i) Let $\alpha\in(\ff12,2)$, $\gamma \in ((1-\alpha)_+, 1]$ and $g\in\bC^\gamma$, $(\mathbf{H}^\gamma)$ and \eqref{926-2} hold.

(ii) Let $\alpha\in(1,2)$ and $g$ be Lipschitz continuous, $(\mathbf{H}^1)$ and \eqref{mon} hold.

\noindent Then for $f=\int^\infty_0P_t\bar g\dif t$ with $\bar g:=g-\mu(g)$,
it holds that $\sL f=-\bar g$.
\et
\begin{proof}
We only consider the first case . Fix $g\in\bC^\gamma$. Let
$
g_n(x)=g(x)\chi_n(x),
$
where $\chi_n:=\chi(x/n)$ and $\chi$ is a smooth cut-off function with $\chi|_{B_1}=1$ and $\chi|_{B^c_2}=0$, where $B_r:=\{x\in\mR^d, |x|\leq r\}$.
By definition,
\begin{align}\label{js10}
\|g_n\|_{\bC^{\gamma}}\lesssim\|g\|_{\bC^\gamma}.
\end{align}
Let $f_n(x):=\int^\infty_0P_t\bar g_n(x)\dif t$ and $f(x):=\int^\infty_0P_t\bar g(x)\dif t$. By Theorem \ref{Th25} and \eqref{js10}, we have for any $\eps\in(0,\gamma)$,
\begin{align*}
\|f\|_{\bC^{1}_{\rm log}}+\|f\|_{\bC^{\gamma+\alpha-\eps}_{\rm log}}+\sup_{n\in\mN}\|f_n\|_{\bC^{1}_{\rm log}}+\sup_{n\in\mN}\|f_n\|_{\bC^{\gamma+\alpha-\eps}_{\rm log}}\lesssim \|g\|_{\bC^\gamma},
\end{align*}
and by \eqref{stein:s1},
$$
\sL f_n(x)=-\bar g_n(x).
$$
By \eqref{Def1}, it suffices to show that for each $x\in\mR^d$,
\begin{align}\label{DQ1}
\lim_{n\to\infty}\sL^{(\alpha)}_{\sigma} f_n(x)=\sL^{(\alpha)}_{\sigma} f(x)
\end{align}
and
\begin{align}\label{DQ2}
 \lim_{n\to\infty}\nabla f_n(x)=\nabla f(x).
\end{align}
By \eqref{Reg0} and the dominated convergence theorem, we have for each $x\in\mR^d$,
$$
\lim_{n\to\infty}\sL^{(\alpha)}_{\sigma}(f_n-f)(x)
=\int^\infty_0\lim_{n\to\infty}\sL^{(\alpha)}_{\sigma} P_t(\bar g_n-\bar g)(x)\dif t
=\int^\infty_0\lim_{n\to\infty}\sL^{(\alpha)}_{\sigma} P_t(g_n-g)(x)\dif t,
$$
and for each $t>0$,
$$
\lim_{n\to\infty}|\sL^{(\alpha)}_{\sigma} P_t(g_n-g)(x)|
\lesssim\int_{\mR^d}\frac{\lim_{n\to\infty}|\delta^{(2)}_yP_{t}{(g_n-g})(x)|}{|y|^{d+\alpha}}\dif y=0.
$$
Thus we get \eqref{DQ1}.
Similarly, we can show \eqref{DQ2}.
The proof is complete.
\end{proof}

\section{Proof of main results}\label{proof}

In this section, we prove the main results by Theorems \ref{Th216}, \ref{92-8} and \ref{Th25}.

\begin{proof}[Proof of Theorem \ref{main}]
Let $\xi\sim\mu_\eta$ and ${\bf d}(x,y)=|x-y|$. Recall that $Z^x_t=x+t b(x)+\sigma(x)L^{(\alpha)}_t$.
By $(\mathbf{H}^1)$, it follows that for any $\beta \in [0, 1]$ and $t \in [0, 1]$,
\begin{align}\label{FD4}
\mE|Z^\xi_t-\xi|^\beta\leq t^\beta\mE|b(\xi)|^\beta+\|\sigma\|^\beta_\infty\mE|L^{(\alpha)}_t|^\beta
\lesssim t^\beta+t^{\beta/\alpha}.
\end{align}
By \eqref{DP2}, we have
\begin{align}\label{FD9}
\cW_{\bf d}(\mu,\mu_\eta)
=\sup_{g \in {\rm Lip}({\bf d})}|\mu_\eta(\sL f)|=\sup_{g\in{\rm Lip({\bf d})}}\left|\mE\sL f(\xi)\right|.
\end{align}
Clearly, by \eqref{FD7}, Theorem \ref{Le23}, and \eqref{FD4}, one can derive certain convergence rate.
Below, we shall use an alternative approach to derive better convergence rate.
For a given $g\in{\rm Lip({\bf d})}$, let $f$ be the solution of Stein/Poisson equation $\sL f=\mu(g)-g$.
For each $s\in[0,\eta]$, we have
$$
g(x)-g(Z^x_s)=\sL f(x)-\sL f(Z^x_s)=\sL f(x)-\sL_{Z^x_s} f(Z^x_s).
$$
Integrating both sides from $0$ to $\eta$ and then subtracting \eqref{DS9}, we obtain
$$
\eta\sL f(x)
=\int^\eta_0[g(x)-\mE g(Z^x_s)]\dif s+[\mE f(Z^x_\eta)-f(x)]+\int^\eta_0\mE\left[\sL_{Z^x_s} f(Z^x_s)-\sL_x f(Z^x_s)\right]\dif s.
$$
Hence, for any $x\in\mR^d$,
\begin{align}
\sL f(x)
&=\frac1\eta\int^\eta_0[g(x)-\mE g(Z^x_s)]\dif s+\frac{\mE f(Z^x_\eta)-f(x)}\eta\no\\
&\quad+\frac1\eta\int^\eta_0\mE\left[\sL^{(\alpha)}_{\sigma(Z^x_s)} f(Z^x_s)-\sL^{(\alpha)}_{\sigma(x)} f(Z^x_s)\right]\dif s\no\\
&\quad+\frac1\eta\int^\eta_0\mE\left[(b(Z^x_s)-b(x))\cdot\nabla f(Z^x_s)\right]\dif s\no\\
&=:I_1(x)+I_2(x)+I_3(x)+I_4(x).\label{II9}
\end{align}
For $I_1(\xi)$, by \eqref{FD4} and $\|\nabla g\|_\infty\leq 1$, we have
$$
|\mE I_1(\xi)|\leq \frac{\|\nabla g\|_\infty}\eta\int^\eta_0\mE|Z^\xi_s-\xi|\dif s
\lesssim \eta+\eta^{1/\alpha}.
$$
For $I_2(\xi)$, since $\xi\sim\mu_\eta$ is the stationary distribution, we have $Z^\xi_\eta\sim\xi$, and so
$$
\mE I_2(\xi)=0.
$$
For $I_3(\xi)$ and $I_4(\xi)$,  by $(\mathbf{H}^1)$ and Theorem \ref{Le23}, we have
\begin{align*}
\mE I_3(\xi)|\lesssim\frac{\|\nabla f\|_{\infty}+\|\nabla^2f\|_{\infty}}\eta\int^\eta_0
\mE|Z^\xi_s-\xi|\dif s\overset{\eqref{FD3}}\lesssim \eta+\eta^{1/\alpha}
\end{align*}
and
\begin{align*}
|\mE I_4(\xi)|\leq \frac{\|\nabla b\|_{\infty}\|\nabla f\|_{\infty}}\eta\int^\eta_0\mE|Z^\xi_s-\xi|\dif s
\overset{\eqref{FD3}}\lesssim \eta+\eta^{1/\alpha}.
\end{align*}
Combining the above estimates, we obtain
$$
|\mE\sL f(\xi)|\lesssim \eta+\eta^{1/\alpha}\lesssim \eta^{1/\alpha}.
$$
Substituting this into \eqref{FD9}, we conclude the proof.
\end{proof}

Now we are in a position to give
\begin{proof}[Proof of Theorem \ref{main1}]
Let ${\bf d}(x,y)=|x-y|^\gamma\wedge 1$ and
fix $g\in{\rm Lip({\bf d})}$.
Let $f$ be the solution of Stein/Poisson equation $\sL f=\mu(g)-g$. Without loss of generality, we assume $g(0)=0$. Otherwise, we use $\tilde g(x)=g(x)-g(0)$ to replace $g$. 
Since $g\in{\rm Lip({\bf d})}$, we clearly have
\begin{align}\label{FD6}
\|g\|_{\bC^\gamma}=\|g\|_\infty+\sup_{|x-y|\leq 1}|g(x)-g(y)|/|x-y|^\gamma\leq 2.
\end{align}
Following the proof of Theorem \ref{main}, letting $\xi\sim\mu_\eta$, we need to estimate $|\mE(\sL f(\xi))|$.
For fixed $R>0$ and $\beta>0$, we have
\begin{align}\label{FD3}
|\mE(\sL f(\xi))|\leq|\mE(\sL f(\xi)\b1_{|\xi|\leq R})|+2\mP(|\xi|>R).
\end{align}
Next we devote to estimate the first term.
By \eqref{II9}, we have
$$
|\mE(\sL f(\xi)\b1_{|\xi|\leq R})|\leq\sum_{i=1}^4\left|\mE \left(I_i(\xi)\b1_{|\xi|\leq R}\right)\right|.
$$
Below, for simplicity of notation, we write
$$
\Lambda_s:=\log(\e+|Z^{\xi}_s|)+\log(\e+|\xi|).
$$
Clearly, we have for any $\beta\in[0,\alpha)$,
$$
\mE|\xi|^\beta<\infty,\quad \mE|L^{(\alpha)}_s|^\beta\lesssim s^{\beta/\alpha},\ \ \mE\Lambda_s^p\lesssim1,\ \
\forall s\in[0,1],\ \  \ \forall p>0.
$$
From this, by H\"older's inequality, it is easy to see that for any $\beta\in[0,\alpha)$,
\begin{align}\label{HF2}
\mE(\Lambda_s |\xi|^\beta)\lesssim 1,\ \ \mE(\Lambda_s |L^{(\alpha)}_s|^\beta)\lesssim s^{\beta/\alpha},\ \ \forall s\in[0,1].
\end{align}
For the term containing $I_1$,  by \eqref{FD6} and \eqref{FD4}, we have  for any $\beta\in(0,1]\cap(0,\alpha)$,
$$
\left|\mE \left(I_1(\xi)\b1_{|\xi|\leq R}\right)\right|
\leq \frac{\|g\|_{\bC^{\gamma\wedge\beta}}}\eta\int^\eta_0\mE\Big(|Z^\xi_s-\xi|^{\gamma\wedge\beta}\Big)
\dif s\lesssim \eta^{\gamma\wedge\beta}+\eta^{\ff{\gamma\wedge\beta}{\alpha}}.
$$
For the term containing $I_2$,   by the ergodicity of EM \eqref{EM1-1}, noting that
$$
\mE f(Z^{\xi}_\eta)=\mE f(Y_1)=\mE f(\xi),
$$
we have
$$
\mE \Big((f(Z^{\xi}_\eta)-f(\xi))\b1_{|\xi|\leq R}\Big)=-\mE \Big((f(Z^{\xi}_\eta)-f(\xi))\b1_{|\xi|>R}\Big).
$$
Hence, for any $\beta\in(0,\alpha)$,
\begin{align*}
\left|\mE \left(I_2(\xi)\b1_{|\xi|\leq R}\right)\right|&\leq\eta^{-1}\mE\Big(|f(Z^{\xi}_\eta)-f(\xi)|\b1_{|\xi|>R}\Big)
\leq\eta^{-1}\|f\|_{L^\infty_{\rm log}}\mE\Big(\Lambda_\eta\b1_{|\xi|>R}\Big)\\
&\overset{\eqref{926-4}}\lesssim\eta^{-1}\|g\|_{\bC^\gamma}\mE\Big(\Lambda_\eta |\xi|^\beta\Big)/R^{\beta}
\overset{\eqref{FD6}, \eqref{HF2}}\lesssim \eta^{-1}R^{-\beta}.
\end{align*}
For the term containing $I_3$,   we have for any $\zeta\in((\alpha-\gamma)_+,\alpha)$
and $\beta\in(0,\alpha)$,
\begin{align*}
\left|\mE \left[I_3(\xi)\b1_{|\xi|\leq R}\right]\right|
&\leq\frac1\eta\int^\eta_0\mE\left|\sL^{(\alpha)}_{\sigma(Z^{\xi}_s)} f(Z^{\xi}_s)-\sL^{(\alpha)}_{\sigma(\xi)} f(Z^{\xi}_s)\right|\dif s\\
&\overset{\eqref{DS5}}\lesssim\frac{\|f\|_{\bC^{\zeta+\gamma}_{\rm log}}}{\eta}\int^\eta_0\mE\left(\Lambda_s
(|Z^{\xi}_s-\xi|^\gamma\wedge 1)\right)\dif s\\
&\overset{\eqref{926-4}}\lesssim\frac{\|g\|_{\bC^{\gamma}}}\eta\int^\eta_0\mE\left[\Lambda_s
(|sb(\xi)|^{\gamma\wedge\beta}+|L^{(\alpha)}_s|^{\gamma\wedge\beta})\right]\dif s
\overset{\eqref{HF2}}\lesssim\eta^{\gamma\wedge\beta}+\eta^{\ff{\gamma\wedge\beta}{\alpha}}.
\end{align*}
Now let us treat the trouble term containing $I_4$, where the trouble comes from the lienar growth of $b$ and the nonexistence of the moments in the supercritical case $\alpha\in(\frac12,1)$.
We make the further decomposition
\begin{align*}
\mE \left(I_3(\xi)\b1_{|\xi|\leq R}\right)&=\frac1\eta\int^\eta_0\mE\left[(b(Z^{\xi}_s)-b(\xi+s b(\xi)))\cdot\nabla f(Z^{\xi}_s)\b1_{\{|\xi|\leq R, |\sigma(\xi)L^{(\alpha)}_s|\leq 1\}}\right]\dif s\\
&+\frac1\eta\int^\eta_0\mE\left[(b(Z^{\xi}_s)-b(\xi+s b(\xi)))\cdot\nabla f(Z^{\xi}_s)\b1_{\{|\xi|\leq R, |\sigma(\xi) L^{(\alpha)}_s|>1\}}\right]\dif s\\
&+\frac1\eta\int^\eta_0\mE\left[(b(\xi+s b(\xi))-b(\xi))\cdot\nabla f(Z^{\xi}_s)\b1_{\{|\xi|\leq R\}}\right]\dif s\\
&=:\cJ_{1}(R)+\cJ_{2}(R)+\cJ_{3}(R).
\end{align*}
For $\cJ_{1}(R)$, by $(\mathbf{H}^\gamma)$, we have  for any $\beta\in(0,1]\cap(0,\alpha)$,
\begin{align*}
|\cJ_{1}(R)|&\lesssim \frac{\|f\|_{\bC^{1}_{\rm log}}}\eta\int^\eta_0\mE\left[\Lambda_s|\sigma(\xi)L^{(\alpha)}_s|^{\gamma\wedge\beta}\right]\dif s
\overset{\eqref{926-4}}\lesssim \frac{\|g\|_{\bC^\gamma}}\eta\int^\eta_0\mE\left[\Lambda_s|L^{(\alpha)}_s|^{\gamma\wedge\beta}\right]\dif s\overset{\eqref{HF2}}\lesssim \eta^{\ff{\gamma\wedge\beta}{\alpha}}.
\end{align*}
For $\cJ_{2}(R)$, noting that by Stein's equation,
$$
(b\cdot\nabla f)(x)=\bar g(x)-\sL^{(\alpha)}_{\sigma}f(x),
$$
by \eqref{926-4} again, we have
$$
|(b\cdot\nabla f)(x)|\leq \|g\|_\infty+\left|\sL^{(\alpha)}_{\sigma}f(x)\right|
\overset{\eqref{DS4}}\lesssim \|g\|_\infty+\|\sL^{(\alpha)}_{\sigma}f\|_{L^\infty_{\log}}\log(\e+|x|)
\overset{\eqref{926-4}}\lesssim\|g\|_{\bC^\gamma}\log(\e+|x|).
$$
Hence, by \eqref{926-4} again, for any $\beta\in(0,1]\cap(0,\alpha)$,
\begin{align*}
|\cJ_{2}(R)|&\leq \frac1\eta\int^\eta_0\mE\left[\Big(|b(\xi+s b(\xi))\cdot\nabla f(Z^{\xi}_s)|+|b\cdot\nabla f|(Z^{\xi}_s)\Big)\b1_{\{|\xi|\leq R, |\sigma(\xi)L^{(\alpha)}_s|>1\}}\right]\dif s\\
&\lesssim \frac{\|g\|_{\bC^\gamma}}\eta\int^\eta_0 \mE \left[\log(\e+|Z^{\xi}_s|)(1+|\xi|)\b1_{\{|\xi|\leq R, |\sigma(\xi)L^{(\alpha)}_s|>1\}}\right]\dif s\\
&\lesssim\frac{1}\eta\int^\eta_0 \mE \left[\Lambda_s(1+|\xi|^\beta R^{1-\beta})|\sigma(\xi)L^{(\alpha)}_s|^\beta\right]\dif s
\overset{\eqref{HF2}}\lesssim R^{1-\beta}\eta^{\beta/\alpha}+\eta^{\beta/\alpha},
\end{align*}
where we have used that $\xi$ is independent with $(L^{(\alpha)}_s)_{s\in[0,\eta]}$ and
$$
(1+|x|)\b1_{\{|x|\leq R, |y|>1\}}
\leq (1+|x|^\beta R^{1-\beta})|y|^\beta.
$$
For $\cJ_{3}(R)$, noting that for $\beta\in(0,1]\cap(0,\alpha)$,
\begin{align*}
&|b(x+sb(x))-b(x)|\b1_{|x|\leq R}
\leq\kappa_0(|sb(x)|^{\gamma\wedge\beta}+|sb(x)|)\b1_{|x|\leq R}\\
&\quad\lesssim s^{\gamma\wedge\beta}(1+|x|^{\gamma\wedge\beta})+s (1+|x|)\b1_{|x|\leq R}
\leq s^{\gamma\wedge\beta}(1+|x|^{\gamma\wedge\beta})+s (1+R^{1-\beta}|x|^\beta),
\end{align*}
we have
\begin{align*}
|\cJ_{3}(R)|
&\lesssim \frac{\|f\|_{\bC^{1}_{\rm log}}}\eta\int^\eta_0
 \mE \left[\Lambda_s(s^{\gamma\wedge\beta}(1+|\xi|^{\gamma\wedge\beta})+s(1+R^{1-\beta}|\xi|^\beta))\right]\dif s
\overset{\eqref{HF2}}\lesssim \eta^{\gamma\wedge\beta}+\eta R^{1-\beta}.
\end{align*}
Combining the above calculations, we obtain that for any $\beta\in(0,1]\cap(0,\alpha)$,
\begin{align*}
\left|\mE(\sL f(\xi)\b1_{|\xi|\leq R})\right|
\lesssim \eta^{\gamma\wedge\beta}+\eta^{\ff{\gamma\wedge\beta}{\alpha}}+\eta^{-1} R^{-\beta}
+ R^{1-\beta}\eta^{\beta/\alpha}+\eta^{\beta/\alpha}+\eta R^{1-\beta},
\end{align*}
which together with \eqref{FD3} yields that
$$
\cW_{\bf d}(\mu,\mu_\eta)
\lesssim \eta^{\gamma\wedge\beta}+\eta^{\ff{\gamma\wedge\beta}{\alpha}}+\eta^{-1} R^{-\beta}
+ R^{1-\beta}\eta^{\beta/\alpha}+\eta^{\beta/\alpha}+\eta R^{1-\beta}+R^{-\beta}\mE|\xi|^\beta.
$$
When $\alpha\in(1,2)$, one chooses $\beta=1$ and $R=\eta^{-1-\frac\gamma\alpha}$, then obtain
$$
\cW_{\bf d}(\mu,\mu_\eta)\lesssim\eta^{\gamma/\alpha}.
$$
When $\alpha\in(\frac12,1]$, for any $\beta\in(0,\alpha)$, one chooses $R=\eta^{-1-\frac\beta\alpha}$, then obtain
$$
\cW_{\bf d}(\mu,\mu_\eta)\lesssim\eta^{\gamma\wedge\beta}+\eta^{\ff{\gamma\wedge\beta}{\alpha}}
+\eta^{\beta(1+\frac\beta\alpha)-1}.
$$
The proof is complete by taking $\beta$ close to $\alpha$.
\end{proof}

\begin{appendix}

\setcounter{equation}{0}

\renewcommand{\theequation}{A.\arabic{equation}}

\section{Ergodicity of the EM scheme \eqref{EM1-1}}\label{EMergo}

First of all, we present the following lemma.

\begin{lemma}\label{nol:decay}
Let $m\in(0,\alpha)$ and $\rho_m(x)=(1+|x|^2)^{m/2}$.
For any $d\times d$-matrix $A$ and $x\in\mR^d$, we have
$$
\sL^{(\alpha)}_{A} \rho_m(x)\leq \frac{2\pi^{\ff{d}{2}}}{\Gamma(\ff{d}{2})} \left(\frac{m}{2-\alpha}+\frac{2^{1+\alpha} }{\alpha-m}\right) \|A\|^\alpha\rho_{m-\alpha}(x).
$$
\end{lemma}
\begin{proof}
We assume $A\not=0_{d\times d}$, and divide the proof into two cases.

(i) We first consider $|x|>1$. Write $R:=\tfrac{|x|}{2\|A\|}$.
By \eqref{L:sy}, we  make the following decomposition:
\begin{align*}
\sL^{(\alpha)}_{A} \rho_m(x)
=\frac12\int_{|z|\leq R}\delta^{(2)}_{Az}\rho_m(x)\frac{\dif z}{|z|^{d+\alpha}}
+\frac12\int_{|z|>R}\delta^{(2)}_{Az}\rho_m(x)\frac{\dif z}{|z|^{d+\alpha}}=:I_1+I_2.
\end{align*}
For $I_1$, we need to estimate $\delta^{(2)}_{Az}\rho_m(x)$ for any $|z|\leq R$.
By the calculus, we have
$$
\nabla^2 \rho_m(x)=m(1+|x|^2)^{\frac{m}{2}-2}[(1+|x|^2)\mI-(2-m)x\otimes x].
$$
Since $m<\alpha<2$, we immediately have for any $\xi\in\mathbb{R}^d$,
\begin{align}\label{bs:0}
\langle\xi, \nabla^2 \rho_m(x)\xi\rangle\leq m(1+|x|^2)^{\frac{m}{2}-1}|\xi|^2\leq m |\xi|^2.
\end{align}
By Taylor's expansion, we have for any $|z|\leq R=\tfrac{|x|}{2\|A\|}$,
\begin{align*}
|\delta^{(2)}_{Az}\rho_m(x)|
&=\left|\int^1_0\int^1_{-1}\langle \nabla^2 \rho_m(x+r_1 r_2Az)Az, r_{2}Az\rangle\dif r_1\dif r_2\right|\\
&\leq m|Az|^2\int^1_0\int^1_{-1}r_{2}(1+|x+r_1r_2Az|^2)^{\frac m2-1}\dif r_1\dif r_2\\
&\leq 2m\|A\|^2|z|^2\Big(1+\tfrac{|x|^2}{4}\Big)^{\frac m2-1}.
\end{align*}
Hence, recalling $R=\tfrac{|x|}{2\|A\|}$ and $S_d=2\pi^{\ff{d}{2}}/\Gamma(\ff{d}{2})$, we have
\begin{align*}
I_1&\leq m \|A\|^2\Big(1+\tfrac{|x|^2}{4}\Big)^{\frac m2-1}\int_{|z|\leq R}|z|^{2-d-\alpha}\dif z\\
     &=m\|A\|^2\Big(1+\tfrac{|x|^2}{4}\Big)^{\frac m2-1}S_d \int_0^{R}r^{1-\alpha}\dif r
     \leq \frac{mS_d}{2-\alpha}\|A\|^\alpha |x|^{m-\alpha}.
\end{align*}
For $I_2$, noting that
$$
|a+b|^r\leq |a|^r+|b|^r,\quad \forall a,b\in\mR,\ r\in(0,1],
$$
we have for any $x,z\in\mR^d$
\begin{align}\label{bs:1}
\delta_{z}\rho_m(x)&=|(1+|x+z|^2)^{\frac m2}-(1+|x|^2)^{\frac m2}|\leq ||x+z|^2-|x|^2|^{\ff{m}{2}}\no\\
&\leq (2|x||z|+|z|^2)^{\ff{m}{2}}\leq (|x|^2+2|z|^2)^{\ff{m}{2}}\leq |x|^m+2^{\ff{m}{2}}|z|^m.
\end{align}
Then recalling $R=\tfrac{|x|}{2\|A\|}$, we have
\begin{align*}
I_2&=\int_{|z|>R}\delta_{Az}\rho_m(x)\frac{\dif z}{|z|^{d+\alpha}}
\leq \int_{|z|>R}\Big(|x|^m+2^{\ff{m}{2}}\|A\|^m|z|^m\Big)\frac{\dif z}{|z|^{d+\alpha}}\\
&\leq  S_d\int_{R}^\infty
\Big(|x|^m+2^{\ff{m}{2}}\|A\|^m r^m\Big)r^{-1-\alpha}\dif r
\leq \frac{2^{1+\alpha} S_d}{\alpha-m}\|A\|^\alpha|x|^{m-\alpha}.
\end{align*}
Combining the above calculations, we get for $|x|>1$,
\begin{align*}
\sL^{(\alpha)}_{A} \rho_m(x)
\leq  S_d \left(\frac{m}{2-\alpha}+\frac{2^{1+\alpha} }{\alpha-m}\right) \|A\|^\alpha|x|^{m-\alpha}
\leq  S_d \left(\frac{m}{2-\alpha}+\frac{2^{1+\alpha} }{\alpha-m}\right) \|A\|^\alpha\rho_{m-\alpha}(x).
\end{align*}

(ii) We consider $|x|\leq 1$. By \eqref{L:sy}, noting that
$$
\delta_{Az}\rho_m(x)\overset{\eqref{bs:1}}\leq
1+2^{\ff{m}{2}}|Az|^m\leq
1+2^{\ff{m}{2}}\|A\|^m|z|^m,
$$
we directly get
\begin{align*}
\sL^{(\alpha)}_{A} \rho_m(x)
&=\frac12\int_{|z|\leq \ff{1}{\|A\|}}\delta^{(2)}_{Az}\rho_m(x)\frac{\dif z}{|z|^{d+\alpha}}
+\int_{|z|>\ff{1}{\|A\|}}\delta_{Az}\rho_m(x)\frac{\dif z}{|z|^{d+\alpha}}\\
&\overset{\eqref{bs:0}}\leq m\|A\|^2\int_{|z|\leq\ff{1}{\|A\|}}|z|^2\frac{\dif z}{|z|^{d+\alpha}}
+\int_{|z|>\ff{1}{\|A\|}}\left(1+2^{\ff{m}{2}}\|A\|^m|z|^m\right)\frac{\dif z}{|z|^{d+\alpha}}\\
& \leq m S_d \|A\|^2\int_0^{\ff{1}{\|A\|}}r^{1-\alpha}\dif r
+S_d\int_{\ff{1}{\|A\|}}^\infty \left(1+2^{\ff{m}{2}}\|A\|^m r^m\right)r^{-1-\alpha}\dif r\\
&\leq \ S_d \left(\frac{m}{2-\alpha}+\frac{2^{1+\ff m2} }{\alpha-m}\right) \|A\|^\alpha\leq  S_d \left(\frac{m}{2-\alpha}+\frac{2^{1+\alpha} }{\alpha-m}\right) \|A\|^\alpha\rho_{m-\alpha}(y).
\end{align*}
Combining (i) and (ii), we complete the proof.
\end{proof}

Recalling the EM scheme $(Y^x_k)_{k\geq 0}$ defined in \eqref{EM1-1}, we have
\begin{lemma}\label{W1.1-1L}
Under $(\mathbf{H}^\gamma)$ and \eqref{Dis}, for any $m\in(0,\alpha)$,
there exist $\eta_0=\eta_0(\Theta)\in(0,1)$ and $C_1,C_2, C_3>0$ depending on $\Theta,m,\eta_0$
such that for any $\eta\in(0,\eta_0)$, there is a unique invariant probability measure $\mu_{\eta}$ such that for any $x\in\mR^d$ and $k\in\mN$,
\begin{align}
\mathbb{E}|Y_{k}^{x}|^{m}\leq (1-C_1\eta)^k\rho_{m}(x)+C_2,\label{js0}
\end{align}
and for any $f\in L_m^\infty(\mR^d)$,
\begin{align}
\sup_{\|f/\rho_{m}\|_{\infty}\leq 1}|\mathbb{E}f(Y_{k}^{x})-\mu_{\eta}(f)|\leq C_3 \rho_{m}(x)\e^{-\lambda k},\label{js1}
\end{align}
where $\lambda=-\log(1-C_1\eta)>0$.
\end{lemma}

\begin{proof}
Let $\rho_m(x)=(1+|x|^2)^{m/2}$ with $m\in(0,\alpha)$.
Fix $\eta>0$ and $x\in\mR^d$. We introduce an auxiliary process
$$Z^{\eta,x}_t:=x+\eta b(x)+\sigma(x) L^{(\alpha)}_t,\quad \forall t\in[0,\eta].$$
By It\^o's formula, we have
\begin{align}\label{AA1}
\mE[\rho_m(Z^{\eta,x}_\eta)]=\rho_m(Z^{\eta,x}_0)+\int^\eta_0 \mE\left[\sL^{(\alpha)}_{\sigma(x)}\rho_m(Z^{\eta,x}_s)\right]\dif s.
\end{align}
On the other hand, by the chain rule and $(\mathbf{H}^\gamma)$, we have
\begin{align*}
\rho_m(Z^{\eta,x}_0)&=\rho_m(x)+\int^\eta_0\left\< b(x),\nabla\rho_m(x+rb(x))\right\>\dif r\\
&=\rho_m(x)+m\int^\eta_0\left\< b(x),x+rb(x)\right\>\rho_{m-2}(x+rb(x))\dif r\\
&\leq\rho_m(x)+m\int^\eta_0\left(-\kappa_1|x|^2+\kappa_2+r|b(x)|^2\right)\rho_{m-2}(x+rb(x))\dif r.
\end{align*}
Since $b$ is linear growth, there are $\eta_1>0$ and $K_1>1$ so that for all $r\in[0,\eta_1]$,
$$
r|b(x)|^2\leq \kappa_1|x|^2/2+K_1,
$$
and
$$
K^{-1}_1 \rho_{m-2}(x)\leq \rho_{m-2}(x+rb(x))\leq K_1\rho_{m-2}(x).
$$
Hence, there are $C_1,K_2>0$ and $\eta_0=\eta_0(\Theta)\in(0,1)$ such that for any $\eta\in(0,\eta_0)$,
\begin{align}\label{AA03}
\rho_m(Z^{\eta,x}_0)
&\leq (1-C_1\eta)\rho_m(x)+K_2 \eta.
\end{align}
On the other hand, by Lemma \ref{nol:decay}, there is a $K_3>0$ such that for all $\eta>0$,
\begin{align}\label{AA3}
\int^\eta_0 \mE\left[\sL^{(\alpha)}_{\sigma(x)}\rho_m(Z^{\eta,x}_s)\right]\dif s\leq K_3\eta.
\end{align}
Combining \eqref{AA1}, \eqref{AA03} and \eqref{AA3}, there is a $C_2=C_2(\Theta)>0$ such that for all $\eta\in(0,\eta_0)$ and $x\in\mR^d$,
\begin{align}\label{AA2}
\mE[\rho_m(Z^{\eta,x}_\eta)]\leq (1-C_1\eta)\rho_m(x)+C_2\eta.
\end{align}
Since $Y^x_k$ and $L_{\eta(k+1)}-L_{\eta k}\sim L_\eta$ are independent, we have
\begin{align*}
\mE[\rho_m(Y^x_{k+1})]
&=\mE[\mE[\rho_m(Z^{\eta,y}_\eta)]|_{y=Y^x_k}]
\leq(1-C_1\eta)\mE[\rho_m(Y^x_k)]+C_2\eta
\leq\cdots\\
&\leq (1-C_1\eta)^{k+1}\rho_m(x)+C_2\eta\sum_{n=1}^\infty (1-C_1\eta)^{n},
\end{align*}
which in turn yields \eqref{js0}.

Finally, since $L^{(\alpha)}_\eta$ has a smooth density $q_\eta (z)$ with the following two-sided estimate:
\begin{align}\label{heat}
q_\eta (z)\overset{\eqref{TG:ep}}\asymp\eta (\eta^{\frac{1}{\alpha}}+|z|)^{-d-\alpha},
\end{align}
the exponential ergodicity \eqref{js1} follows by \eqref{js0} and
\cite[Theorem 6.3]{M-T2}.
\end{proof}

\setcounter{equation}{0}

\renewcommand{\theequation}{B.\arabic{equation}}

\section{Explicit convergence rate of EM schemes for OU processes}\label{A:OU}

In this appendix, we consider the one-dimensional Ornstein-Uhlenbeck (OU) process:
\begin{align}\label{ou}
\dif X^x_t=-X^x_t\dif t+\dif L^{(\alpha)}_t,\quad X^x_0=x,
\end{align}
where $\alpha\in(0,2)$, and the corresponding EM scheme: for $\eta\in(0,1)$ and $k \geq 0$,
\begin{align}\label{ou1}
Y^x_{k+1} = Y^x_k -\eta Y^x_k + L^{(\alpha)}_{\eta(k+1)}- L^{(\alpha)}_{\eta k},  \ \ Y^x_0=x,
\end{align}
where $(L^{(\alpha)}_t)_{t\geq 0}$ is a $1$-dimensional symmetric $\alpha$-stable process with characteristic function 
$$
\mE \e^{i zL^{(\alpha)}_1}=\e^{-|z|^\alpha}.
$$
The main result of this section is presented as below:
\begin{proposition}\label{sy}
Let $\mu$, $\mu_\eta$ be the stationary distribution of \eqref{ou} and \eqref{ou1}, respectively. Then
$$
\mu\sim \alpha^{-\ff1\alpha}L^{(\alpha)}_1,\ \ \mu_\eta\sim (\tfrac{\eta}{1-(1-\eta)^{\alpha}})^{\ff1\alpha}L^{(\alpha)}_1.
$$
\end{proposition}

\begin{proof}
(i) Since $\mu$ is the stationary distribution of SDE \eqref{ou}, we have for any $f\in\bC^2$,
$$
\int_\mR \big(\Delta^{\ff\alpha 2}f(x)-x f'(x)\big)\mu(\dif x)=0.
$$
In particular, taking $f(x)=\e^{-i zx}$, we get
\begin{align*}
\int_\mR \Delta^{\ff\alpha 2}\e^{-i zx}\mu(\dif x)=-i z\int_\mR x\e^{-i zx}\mu(\dif x)=z\p_z\int_\mR  \e^{-i zx}\mu(\dif x).
\end{align*}
Let $\hat{\mu}(z)=\int_\mR \e^{-i zx}\mu(\dif x)$ be the Fourier transform of $\mu$.
By \cite[p. 117]{S70}, we have
\begin{align*}
-|z|^{\alpha}\hat{\mu}(z)=\widehat{\Delta^{\ff\alpha 2}\mu}(z)=z  \hat{\mu}(z)',\quad \forall z\in\mR.
\end{align*}
Since $\hat\mu(0)=1$, solving this differential equation, we get
\begin{align}\label{ft1}
\hat{\mu}(z)=\e^{-|z|^\alpha/\alpha}\Rightarrow \mu\sim \alpha^{-\ff1\alpha}L^{(\alpha)}_1.
\end{align}

(ii) Let $Y_0=\xi$ and $\xi\sim\mu_{\eta}$ be independent of $(L^{(\alpha)}_t)_{t\geq 0}$. We have
$$
Y_1=(1-\eta)\xi+L^{(\alpha)}_\eta\sim \mu_\eta,
$$
which implies that
\begin{align*}
\widehat{\mu_\eta}(z)=\mE\[\e^{-iz\xi}\]=\mE\[\e^{-iz((1-\eta)\xi+L^{(\alpha)}_\eta) }\]
=\e^{-\eta |z|^{\alpha}}\mE\[\e^{-iz(1-\eta)\xi }\]
=\e^{-\eta |z|^{\alpha}}\widehat{\mu_\eta}(z-z\eta).
\end{align*}
By induction method, we have for any $m\in\mN$
\begin{align*}
\widehat{\mu_\eta}(z)=\widehat{\mu_\eta}\((1-\eta)^mz\)\prod_{k=0}^m\e^{-\eta(1-\eta)^{\alpha k}|z|^\alpha}
=\widehat{\mu_\eta}\((1-\eta)^mz\)\e^{-\eta|z|^\alpha \ff{1-(1-\eta)^{\alpha m}}{1-(1-\eta)^{\alpha}}}.
\end{align*}
By taking $m\to\infty$, as $\eta\in(0,1)$ and $\widehat{\mu_\eta}(0)=1$, we have
\begin{align}\label{ft2}
\widehat{\mu_\eta}(z)=\e^{- \eta|z|^{\alpha}/(1-(1-\eta)^\alpha)}
\Rightarrow \mu_\eta\sim (\tfrac{\eta}{1-(1-\eta)^{\alpha}})^{\frac1\alpha}L^{(\alpha)}_1.
\end{align}
The proof is complete.
\end{proof}

Based on Proposition \ref{sy}, we can show the following result.
\begin{theorem}\label{syy2}
Let $\alpha\in(1,2)$ and ${\bf d}(x,y)=|x-y|$. It holds that
\begin{align*}
\frac{\alpha-1}{\alpha}\int_{0}^{1}r^{\alpha}\e^{-\ff{r^\alpha}{\alpha}}\dif r\leq \varliminf_{\eta\to0}\frac{\cW_{\bf d}(\mu_\eta,\mu)}\eta\leq \varlimsup_{\eta\to0}\frac{\cW_{\bf d}(\mu_\eta,\mu)}\eta\leq \frac{\alpha-1}{2\alpha^{(1+\alpha)/\alpha}}\mE \big|L^{(\alpha)}_1\big|.
\end{align*}
\end{theorem}

\begin{proof}
Let $\alpha\in(1,2)$. By the definition of $\cW_{\bf d}$ and  Proposition \ref{sy}, we have
\begin{align*}
\cW_{\bf d}(\mu_\eta,\mu)=\ \inf_{(X,Y)\in\Pi(\mu_\eta,\mu)}\mathbb{E}|X-Y|\leq   \left|\big(\tfrac{\eta}{1-(1-\eta)^{\alpha}}\big)^{\frac1\alpha}-\alpha^{-\ff1\alpha}\right|\mE \big|L^{(\alpha)}_1\big|.
\end{align*}
Elementary calculations gives that
$$
\lim_{\eta\to 0}\left|\big(\tfrac{\eta}{1-(1-\eta)^{\alpha}}\big)^{\frac1\alpha}-\alpha^{-\ff1\alpha}\right|/\eta=\tfrac{\alpha-1}{2\alpha^{(1+\alpha)/\alpha}}.
$$
Hence,
\begin{align*}
\varlimsup_{\eta\to0}\cW_{\bf d}(\mu_\eta,\mu)/\eta\leq \tfrac{\alpha-1}{2\alpha^{(1+\alpha)/\alpha}}\mE \big|L^{(\alpha)}_1\big|.
\end{align*}
On the other hand, we introduce the function
\begin{align}\label{B2}
\phi(x)=\int_{-1}^{1}\cos(rx)\dif r.
\end{align}
Clearly, $|\phi(x)-\phi(y)|\leq |x-y|$.
By \eqref{jsy5}, we have
\begin{align*}
\cW_{\bf d}(\mu_\eta,\mu)\geq |\mu_\eta(\phi)-\mu(\phi)|.
\end{align*}
By Fubini's theorem, we have
$$
\mu(\phi)=\int_{\mR}\int_{-1}^{1}\cos(rx)\dif r\mu(\dif x)
=\int_{\mR}\int_{-1}^{1}\e^{-irx}\dif r\mu(\dif x)
=\int_{-1}^{1}\hat{\mu}(r)\dif r\overset{\eqref{ft1}}
=\int_{-1}^{1}\e^{-\ff{|r|^\alpha}{\alpha}}\dif r
$$
and 
$$
\mu_\eta(\phi)
=\int_{\mR}\int_{-1}^{1}\e^{-irx}\dif r\mu_\eta(\dif x)
=\int_{-1}^{1}\widehat{\mu_\eta}(r)\dif r\overset{\eqref{ft2}}
=\int_{-1}^{1}\e^{- \eta|r|^{\alpha}/(1-(1-\eta)^\alpha)}\dif r.
$$
Noting that
\begin{align}\label{BB2}
\lim_{\eta\to0}\Big(\e^{- \eta|r|^{\alpha}/(1-(1-\eta)^\alpha)}-\e^{-\ff{|r|^\alpha}{\alpha}}\Big)\eta^{-1}
=-\ff{\alpha-1}{2\alpha}|r|^{\alpha}\e^{-\ff{|r|^\alpha}{\alpha}},
\end{align}
by the dominated convergence theorem, we get
$$
\varliminf_{\eta\to0}\frac{\cW_{\bf d}(\mu_\eta,\mu)}\eta\geq \lim_{\eta\to0}\frac{|\mu_\eta(\phi)-\mu(\phi)|}\eta=\int_{-1}^{1}\ff{\alpha-1}{2\alpha}|r|^{\alpha}\e^{-\ff{|r|^\alpha}{\alpha}}\dif r.
$$
The proof is complete.
\end{proof}

\end{appendix}

{\bf Funding}: The first author is supported by NNSFC grant of China (No. 12301176), the Natural Science Foundation of Jiangsu Province (No. BK20220867), and the China Postdoctoral
Science Foundation (No. 2023M741692).

The second author is supported by NNSFC grant of China (No. 12071499), the Science and Technology Development Fund of Macau S.A.R. FDCT 0074/2023/RIA2, and the University of Macau grant MYRG-GRG2023-00088-FST.

The third author is supported by NNSFC grant of China (No. 12401172), the China Postdoctoral Science Foundation (No. 2023M740261), and the Postdoctoral Fellowship Program of China Postdoctoral Science Foundation (Class C, No. GZC20242180).

The fourth author is supported by the National Key R\&D program of China (No. 2023YFA1010103) and NNSFC grant of China (No. 12131019)  and the DFG through the CRC 1283 ``Taming uncertainty and profiting from randomness and low regularity in analysis, stochastics and their applications''.

\end{document}